\documentclass[journal,twoside,web]{ieeecolor}
\usepackage{generic}
\usepackage{cite}
\usepackage{amsmath,amssymb,amsfonts}
\usepackage{algorithmic}
\usepackage{graphicx}
\usepackage{algorithm,algorithmic}
\usepackage[caption=false,font=footnotesize]{subfig}
\usepackage{placeins}
\usepackage{hyperref}
\hypersetup{hidelinks}
\usepackage{textcomp}
\usepackage{tikz}
\usetikzlibrary{arrows.meta,positioning}
\usepackage{etoolbox}

\usepackage{pgfplots}
\pgfplotsset{compat=1.18}
\usepackage{tikz}
\usetikzlibrary{arrows.meta}

\hypersetup{hypertexnames=false}
\allowdisplaybreaks[3]
\allowdisplaybreaks[4]
\newtheorem{theorem}{Theorem}
\newtheorem{lemma}{Lemma}
\newtheorem{corollary}{Corollary}
\newtheorem{proposition}{Proposition}

\newtheorem{assumption}{Assumption}
\newtheorem{remark}{Remark}

\newcommand{\plotinclude}[2][]{%
  \IfFileExists{#2}{%
    \includegraphics[#1]{#2}%
  }{%
    \fbox{\parbox[c][1.05in][c]{0.29\textwidth}{\centering
    Missing figure\\[-0.2ex]\footnotesize\texttt{\detokenize{#2}}}}%
  }%
}

\def\BibTeX{{\rm B\kern-.05em{\sc i\kern-.025em b}\kern-.08em
    T\kern-.1667em\lower.7ex\hbox{E}\kern-.125emX}}
\begin{document}
\title{A Continuous-Time Analysis of Smoothed Matrix-Polar Spectral Gradient Flows for Muon-Type Optimization}
\author{Jinlin~Liu, Song~Chen, Jiaxu~Liu, and Chao~Xu,~\IEEEmembership{Senior Member, IEEE}%
\thanks{This work was supported by the Key Research and Development Project of China National Tobacco Corporation under Grant 110202402018 and by the National Natural Science Foundation of China under Grants 62573382, 12171431, 62373323, and 72342025. \emph{(Corresponding authors: Song Chen and Chao Xu.)}}
\thanks{Jinlin Liu is with the School of Mathematical Sciences, Zhejiang University, Hangzhou, Zhejiang 310027, China (e-mail: 3230104798@zju.edu.cn).}%
\thanks{Song Chen is with the Department of Mathematics, National University of Singapore, Singapore (e-mail: song.chen@nus.edu.sg).}%
\thanks{Jiaxu Liu is with the School of Science, Huzhou Normal University, Huzhou 313000, China (e-mail: jiaxuliu@zju.edu.cn).}%
\thanks{Chao Xu is with the State Key Laboratory of Industrial Control Technology, the Institute of Cyber-Systems and Control, College of Control Science and Engineering, Zhejiang University, Hangzhou, Zhejiang 310027, China, and also with Huzhou Institute of Zhejiang University, Huzhou, Zhejiang 313000, China (e-mail: cxu@zju.edu.cn).}}

% column for the text to clear the IEEEpubid mark.


\maketitle


\begin{abstract}
    This paper studies smoothed matrix-polar spectral gradient flows for unconstrained matrix-valued optimization.
    The canonical polar-factor map loses smoothness at rank-deficient matrices and becomes ill-conditioned as singular values approach zero, creating analytical difficulties.
    We therefore introduce a spectral feedback law generated by a smooth spectral potential and establish the regularity, monotonicity, boundedness, and dissipation properties of the feedback.
    Based on this feedback law, we propose a smoothed spectral gradient flow and prove well-posedness and global convergence of the flow.
    We derive convergence-rate results for the spectral gradient flow in nonconvex, convex, and Polyak--Lojasiewicz (PL) settings and analyze the Lyapunov structure and convergence of a momentum-augmented system under the same spectral feedback law. Furthermore, we provide a local descent-rate comparison between the smoothed spectral-gradient direction and the standard Frobenius-gradient direction using a general Hessian-based quadratic model. This analysis yields a verifiable normalized descent-rate advantage condition, showing that the local benefit of the spectral direction depends on both first-order alignment with the gradient matrix and the directional curvature induced by the Hessian.
\end{abstract}


% \fi
\IEEEpeerreviewmaketitle


\section{Introduction}

\IEEEPARstart{N}{eural-network} parameters are naturally organized
as matrices and higher-order tensors, as illustrated in
Fig.~\ref{fig:weight_representation}. A common approach is to
vectorize these parameters and apply optimization methods
developed for Euclidean vector spaces. However, vectorization may obscure the multilinear and
spectral structures inherent in the original representation.
This observation motivates the direct study of optimization over
structured matrix- and tensor-valued variables. From a
continuous-time perspective, it also suggests extending
vector-valued optimization dynamics and feedback laws to
matrix- and tensor-valued systems whose update mechanisms respect
the underlying structure. This structural viewpoint has recently
gained prominence in large-scale deep learning optimization
\cite{bernstein2025modular,liu2025muon}. A representative example
is the Muon optimizer. Specifically, to solve the unconstrained matrix-structured optimization problem $\min_{W\in\mathbb{R}^{m\times n}} f(W)$, a typical Muon-type update can be written as

\begin{figure*}[!t]
    \centering
    \includegraphics[width=0.97\textwidth]{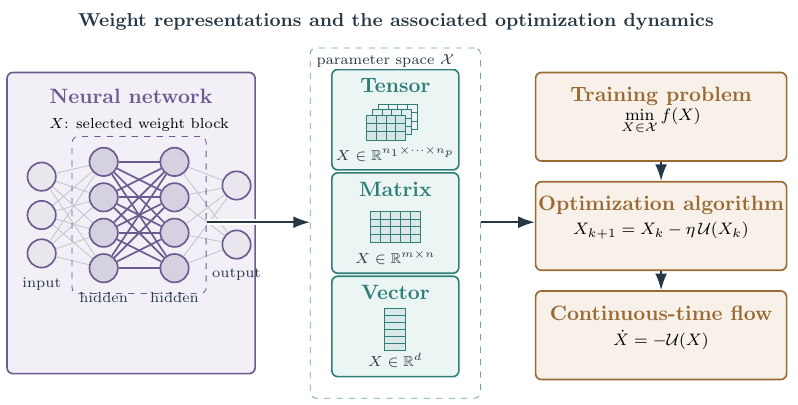}
    \caption{A selected neural-network weight block can be represented in tensor,
matrix, or vector form, leading to a training problem
$\min_{X\in\mathcal X} f(X)$ on the chosen parameter space
$\mathcal X$. From this problem, one may design a discrete
optimization algorithm
$X_{k+1}=X_k-\eta\,\mathcal U(X_k)$,
where $f$ is the loss function and $\mathcal U$ is an update operator related to $\nabla f(X)$.} 
    \label{fig:weight_representation}
\end{figure*}

\begin{align*}
    M_k&=\beta M_{k-1}+\nabla f(W_k),\\
    W_{k+1}&=W_k-\eta\,\operatorname{Polar}(M_k),
\end{align*}
where the matrix polar factor of a full-column-rank matrix \(M\) is
\(\operatorname{Polar}(M)=M(M^\top M)^{-1/2}\). Here, \(M_k\) is a momentum state, \(\eta>0\) is the stepsize, and \(\beta\in[0,1)\) is the momentum parameter. If $\beta=0$, the above scheme reduces to $W_{k+1}=W_k-\eta\,\operatorname{Polar}(\nabla f(W_k))$. Unlike standard gradient descent or common coordinatewise adaptive gradient methods, Muon does not merely adjust the stepsize of each coordinate; instead, it applies a matrix-polar-type normalization to a matrix-valued gradient or momentum state. In practical implementations, the matrix-polar update direction $\operatorname{Polar}(\cdot)$ is often computed approximately, for example by Newton--Schulz iterations or related matrix-iteration schemes \cite{bjorck1971iterative,kim2026newtonschulz}. The distinctive feature of this update is that it preserves the singular-vector structure of the gradient or momentum matrix while normalizing the corresponding singular values. Therefore, a Muon update direction can be interpreted as a spectral gradient direction.

Recent analyses of Muon-type iterations have developed discrete-time
convergence bounds and clarified the role of matrix-polar normalization
under different assumptions \cite{li2025note,shen2025convergence}.

Furthermore, a control-system perspective distinguishes Muon-type methods
from the standard gradient-flow viewpoint. For standard gradient descent,
the classical continuous-time limit is the gradient flow
$\dot W=-\nabla f(W)$
\cite{polyak1964methods,su2016differential,wibisono2016variational,wilson2021lyapunov},
which can also be represented as the feedback control system
\(\dot W=u\), \(u=-\nabla f(W)\), where the feedback term $u$ is given
directly by the Frobenius gradient. This interpretation connects optimization
dynamics with dissipativity and feedback-system analysis
\cite{willems1972dissipativeI,andrieu2010uniting}.
Related continuous-time gradient and consensus dynamics, as well as
decentralized stochastic optimization schemes, have been analyzed using
Lyapunov, invariance, and robustness tools
\cite{gharesifard2014distributed,lu2012zgs,wang2023quantization}.
Classical extremum-seeking feedback provides another control-theoretic route
to real-time optimization \cite{krstic2000extremum,guay2003adaptive}.
Recent feedback-oriented designs include projected zeroth-order dynamics,
backstepping-based accelerated flows, and policy-optimization methods for
robust control \cite{chen2025zerothorder,chen2026backstepping,
guo2023derivativefree}.
Control-inspired optimization frameworks, including PIDAO models
\cite{chen2024pidao}, motivate a natural question: can one construct a
continuous-time spectral gradient-flow model that reflects the Muon feedback
mechanism? The feedback term in Muon-type methods has an explicit spectral
structure, \(u=-\operatorname{Polar}(G)\). This transformation preserves the
singular vectors of the gradient and modifies the distribution of the gradient singular values
\cite{bhatia1997matrix,higham2008functions}. This structure therefore motivates the Muon-type continuous-time flow
\(\dot W=-\operatorname{Polar}(G)\).

Finite-time and fixed-time optimization flows further demonstrate how
the choice of continuous-time feedback law changes the convergence
mechanism \cite{garg2021fixedtime,
lin2017distributed}. The continuous-time optimization literature has also
developed Lyapunov and invariance-based analyses for distributed,
constrained, time-varying, and stochastic optimization problems
\cite{shi2013optimalconsensus,rahili2017timevarying,zeng2017constrained,
charalambous2014extremum}.

However, directly using the exact matrix-polar flow leads to technical difficulties. The mapping
\(G\mapsto G(G^\top G)^{-1/2}\) is nonsmooth near small singular values and rank changes. As a result, exact matrix-polar feedback laws are not immediately suitable for classical ODE-based well-posedness and Lyapunov analysis \cite{higham2008functions,lewis1996convex,lewis2005singular}. To avoid this nonsmoothness while preserving the spectral normalization mechanism of Muon-type methods, we introduce the smoothed matrix-polar feedback law \[h_\epsilon(G)=G(G^\top G+\epsilon I)^{-1/2}, \qquad \epsilon>0.\]
When the singular values of \(G\) are much larger than \(\sqrt{\epsilon}\), the map \(h_\epsilon(G)\) closely approximates the exact matrix-polar direction; when they approach zero, the map smoothly transitions to linear feedback, ensuring the desired regularity for rigorous dynamical analysis.

Based on this feedback law, we first study the direct smoothed spectral gradient flow
\(\dot W=-h_\epsilon(\nabla f(W))\).
Within a feedback-modeling framework, this model can be understood as a Muon-type continuous-time system obtained by replacing ordinary gradient feedback with smoothed matrix-polar spectral feedback. Compared with the standard gradient flow, this system descends along a spectrally normalized gradient direction. Compared with the discrete Muon algorithm, the model removes the finite stepsize, Newton--Schulz approximation, and nonsmooth matrix-polar singularities, thereby enabling continuous-time Lyapunov analysis and local curvature-based comparison.

We then analyze the basic dynamical properties of the flow, including well-posedness,
energy dissipation, boundedness, and convergence to the stationary set.
These properties clarify whether the smoothed spectral feedback law
defines a stable optimization flow in the classical ODE sense.

In addition to the direct spectral gradient flow, we also consider a momentum-augmented system
\[
\dot W=-h_\epsilon(M),
\qquad
\dot M=a\nabla f(W)-bM,
\qquad a>0,\ b>0.
\]
This system mirrors the momentum-state structure of Muon-type methods. We show that the system admits a natural Lyapunov energy function and analyze the long-time convergence behavior of the system. The proposed continuous-time system is not claimed to be an
exact finite-stepsize equivalent of the discrete Muon algorithm.
Rather, it is studied as a smoothed Muon-type spectral feedback
flow that reveals the continuous-time structure and dissipation
mechanism of matrix-polar update directions. Nevertheless, on
every fixed finite time interval over which the relevant gradient
or momentum singular values remain uniformly bounded away from
zero, the smoothed flow converges uniformly to the corresponding
ideal matrix-polar flow as $\epsilon\to 0$.

Beyond these global dynamical properties, we also use the continuous-time
viewpoint to explain a local mechanism behind spectral normalization.
This question is related to recent structural interpretations of gradient
orthogonalization, norm-constrained linear minimization oracles, and
layerwise spectral conditions \cite{kovalev2025orthogonalization,davis2025spectral}.
At a fixed point \(W\), a natural question is when the smoothed spectral
direction can be more efficient than the standard Frobenius-gradient
direction under the local second-order loss structure. To address this pointwise
question, we compare the two directions under a general Hessian-based
quadratic model subject to positive directional-curvature conditions. This
yields a verifiable local advantage criterion, which characterizes the
spectral and curvature conditions under which the smoothed matrix-polar update
direction can provide a larger local descent rate.

The main contributions of this paper are summarized as follows. 
\begin{itemize}
    \item First, we introduce a smoothed matrix-polar spectral feedback law and establish the key properties of the feedback law, including regularity, monotonicity, boundedness, and sector-type dissipation.
    \item Second, we prove global well-posedness and convergence of the
    direct spectral flow and derive nonconvex, convex, and
    PL convergence-rate bounds.
    \item Third, we establish a compatible Lyapunov structure and
    convergence-rate bounds for a momentum-augmented flow, and
    prove a finite-horizon approximation result for the ideal matrix-polar
    flows under a uniform lower singular-value bound.
    \item Finally, we provide a local descent-rate comparison between the smoothed spectral update direction and the Frobenius-gradient update direction, yielding a verifiable local advantage criterion.
\end{itemize}

The rest of the paper is organized as follows. Section II introduces the
smoothed spectral feedback law and the basic properties of the law. Section III
studies the direct smoothed spectral gradient flow and the
momentum-augmented system. Section IV establishes convergence-rate results
for the direct and the momentum-augmented systems. Section V develops a local
descent-rate comparison between the smoothed spectral update direction and the
Frobenius-gradient update direction, providing a pointwise explanation of when
spectral normalization improves local descent efficiency.
Section VI presents numerical experiments.

\textit{Notation.}
For a matrix \(A\in\mathbb R^{m\times n}\), \(A^\top\) denotes the transpose of \(A\), and
\(\operatorname{Tr}(A)\) denotes the trace of \(A\) when \(A\) is square. The symbol \(I\) denotes an identity matrix of compatible dimension. The
Frobenius inner product and Frobenius norm are defined by \(\langle A,B\rangle=\operatorname{Tr}(A^\top B)\) and \(\|A\|_F=\sqrt{\langle A,A\rangle}=(\sum_{i,j} A_{ij}^2)^{1/2}\). We use a singular value decomposition \(A=U\operatorname{diag}(\sigma_1,\ldots,\sigma_d)V^\top\), where \(d=\min\{m,n\}\),
\(\sigma_i=\sigma_i(A)\ge 0\) are the singular values, and
\(
U\in\mathbb R^{m\times d},
V\in\mathbb R^{n\times d}
\). 
The notation
\(\operatorname{diag}(\sigma_1,\ldots,\sigma_d)\) denotes the diagonal
matrix with diagonal entries \(\sigma_1,\ldots,\sigma_d\). The nuclear norm
of \(A\), defined as the sum of the singular values, is
\(\|A\|_*=\sum_{i=1}^d \sigma_i(A)\). The Frobenius norm also satisfies
\(\|A\|_F^2=\sum_{i=1}^d \sigma_i(A)^2\).

For a nonempty set \(\mathcal A\subset \mathbb R^{m\times n}\), the
distance from \(W\in \mathbb{R}^{m\times n}\) to \(\mathcal A\) is measured using the Frobenius norm and is
defined by \(\operatorname{dist}(W,\mathcal A)=\inf_{Y\in\mathcal A}\|W-Y\|_F\).
For a trajectory \(W(t)\) initialized at \(W_0\), the omega-limit set of the trajectory is
defined as
\[
\begin{aligned}
\omega(W_0)=\bigl\{\bar W\in\mathbb R^{m\times n}:{}&
\exists\,t_k\to\infty\ \text{such that}\\
&W(t_k)\to\bar W\bigr\}.
\end{aligned}
\]
For an extended trajectory \((W(t),M(t))\), the notation
\(\omega(W_0,M_0)\) is defined analogously in the product space.

When \(f\) is \(C^2\), \(\mathcal{H}_W[D]\) denotes the Hessian action of \(f\)
at \(W\) along the direction \(D\), namely
\[
\mathcal{H}_W[D]
=
\left.
\frac{d}{ds}\nabla f(W+sD)
\right|_{s=0}.
\]
It describes the first-order variation of the gradient when \(W\) is
perturbed along \(D\). In particular,
\[
\langle D,\mathcal{H}_W[D]\rangle
=
\left.
\frac{d^2}{ds^2}f(W+sD)
\right|_{s=0}
\]
represents the local curvature of the objective along \(D\). As shown in Figure \ref{fig:hessian_action_surface}, a positive
value indicates that the objective is locally convex along this direction,
while a larger value corresponds to stronger curvature.

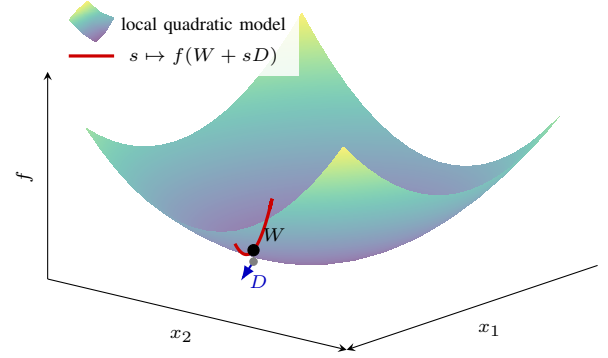
\begin{figure}[!t]
\centering
\begin{tikzpicture}
\pgfmathsetmacro{\wx}{0.55}
\pgfmathsetmacro{\wy}{-0.45}
\pgfmathsetmacro{\dx}{0.85}
\pgfmathsetmacro{\dy}{0.48}
\pgfmathsetmacro{\wz}{0.6*(\wx)^2 + 0.2*(\wx)*(\wy) + 0.9*(\wy)^2}

\begin{axis}[
    width=\columnwidth,
    height=0.72\columnwidth,
    view={130}{28},
    axis lines=left,
    xlabel={$x_1$},
    ylabel={$x_2$},
    zlabel={$f$},
    xlabel style={font=\scriptsize},
    ylabel style={font=\scriptsize},
    zlabel style={font=\scriptsize},
    domain=-1.7:1.7,
    y domain=-1.7:1.7,
    samples=18,
    samples y=18,
    colormap/viridis,
    z buffer=sort,
    grid=none,
    clip=false,
    enlargelimits=0.08,
    xtick=\empty,
    ytick=\empty,
    ztick=\empty,
    legend style={
        at={(0.02,0.98)},
        anchor=north west,
        font=\scriptsize,
        draw=none,
        fill=white,
        fill opacity=0.86,
        text opacity=1
    }
]

\addplot3[surf, opacity=0.58, shader=interp]
{0.6*x^2 + 0.2*x*y + 0.9*y^2};
\addlegendentry{local quadratic model}

\addplot3[very thick, red!80!black, samples=70, domain=-1.05:1.05]
(
    {\wx + x*\dx},
    {\wy + x*\dy},
    {0.6*(\wx + x*\dx)^2 + 0.2*(\wx + x*\dx)*(\wy + x*\dy) + 0.9*(\wy + x*\dy)^2}
);
\addlegendentry{$s\mapsto f(W+sD)$}

\addplot3[only marks, mark=*, mark size=2.1pt, black]
coordinates {(\wx,\wy,\wz)};

\node[font=\scriptsize, anchor=south west] at (axis cs:\wx,\wy,\wz) {$W$};

\addplot3[only marks, mark=*, mark size=1.4pt, gray]
coordinates {(\wx,\wy,0)};

\addplot3[densely dashed, gray]
coordinates {(\wx,\wy,0) (\wx,\wy,\wz)};

\addplot3[-{Latex[length=2mm]}, thick, blue!75!black]
coordinates {(\wx,\wy,0) ({\wx+0.72*\dx},{\wy+0.72*\dy},0)};

\node[font=\scriptsize, anchor=west, blue!75!black]
    at (axis cs:{\wx+0.75*\dx},{\wy+0.75*\dy},0) {$D$};

\end{axis}
\end{tikzpicture}
\caption{Illustration of the Hessian action and directional curvature.
At the point $W$, the direction \(D\) determines the slice
\(s\mapsto f(W+sD)\) shown by the red curve. The quantity
\(\langle D,\mathcal H_W[D]\rangle
=
\left.\frac{d^2}{ds^2}f(W+sD)\right|_{s=0}\)
is the second derivative of this slice at \(s=0\), and hence measures the
local curvature of the objective along the direction \(D\).}
\label{fig:hessian_action_surface}
\end{figure}

	\section{Smoothed Spectral Feedback Law}
\label{sec:spectral_feedback}

For $\epsilon>0$, define the smoothed matrix-polar feedback law
\begin{align}
	h_\epsilon(Z)
	=
	Z(Z^\top Z+\epsilon I)^{-1/2}.
	\label{eq:h_epsilon_def}
\end{align}
If $Z=U\operatorname{diag}(\sigma_i)V^\top$ is a singular value
decomposition of $Z$, then
\begin{align}
	h_\epsilon(Z)
	=
	U\operatorname{diag}
	\left(
	\frac{\sigma_i}{\sqrt{\sigma_i^2+\epsilon}}
	\right)V^\top .
	\label{eq:h_epsilon_svd}
\end{align}
Thus, $h_\epsilon$ preserves the singular vectors of $Z$ and smoothly
normalizes the singular values. It approaches the standard matrix-polar direction
on singular modes satisfying $\sigma_i\gg\sqrt{\epsilon}$ and is
linear near zero.

The feedback law is generated by the spectral energy
\begin{align}
	\Phi_\epsilon(Z)
	=
	\operatorname{Tr}
	\left[
	(Z^\top Z+\epsilon I)^{1/2}
	-
	\sqrt{\epsilon}I
	\right],
	\label{eq:Phi_epsilon_def}
\end{align}
whose gradient satisfies
\begin{align}
	\nabla\Phi_\epsilon(Z)=h_\epsilon(Z).
	\label{eq:h_epsilon_gradient}
\end{align}
Let $d=\min\{m,n\}$. The following properties are used throughout the
paper.

\begin{lemma}
	\label{lem:heps_properties}
For any $\epsilon>0$ and any $M,N\in \mathbb R^{m\times n}$, the map
$h_\epsilon:\mathbb R^{m\times n}\to\mathbb R^{m\times n}$ satisfies the following properties:

\noindent\textup{(i)} It is continuously differentiable and globally
	Lipschitz, with
	\begin{align}
		\|h_\epsilon(M)-h_\epsilon(N)\|_F
		\le
		\frac1{\sqrt{\epsilon}}\|M-N\|_F .
		\label{eq:heps_lipschitz}
	\end{align}

	\noindent\textup{(ii)} It is generated by the spectral potential:
	\begin{align}
		\nabla_M\Phi_\epsilon(M)=h_\epsilon(M).
		\label{eq:Phi_gradient}
	\end{align}

	\noindent\textup{(iii)} It is monotone:
	\begin{align}
		\left\langle
		h_\epsilon(M)-h_\epsilon(N),M-N
		\right\rangle
		\ge0.
		\label{eq:heps_monotone}
	\end{align}

	\noindent\textup{(iv)} It is uniformly bounded:
	\begin{align}
		\|h_\epsilon(M)\|_F\le\sqrt d.
		\label{eq:heps_bound}
	\end{align}

	\noindent\textup{(v)} If $\|M\|_F\le K$, then
	\begin{align}
		\frac{\|M\|_F^2}{\sqrt{K^2+\epsilon}}
		\le
		\langle h_\epsilon(M),M\rangle
		\le
		\frac{\|M\|_F^2}{\sqrt{\epsilon}}.
		\label{eq:heps_dissipation_bounds}
	\end{align}
\end{lemma}

The proof is provided in Appendix~\ref{app:heps_properties}. These structural properties make \(h_\epsilon\) suitable as a feedback law in continuous-time optimization dynamics: the Lipschitz estimate gives well-posedness, while the monotonicity and sector dissipation estimates provide the basic Lyapunov mechanism used below.

\section{Smoothed Spectral Gradient Flow}
\label{sec:direct_flow}
\label{sec:wellposed_momentum}
We now turn to the induced smoothed spectral gradient dynamics.
\subsection{Direct Spectral Gradient Flow}
We derive the direct smoothed spectral gradient flow using the feedback map \eqref{eq:h_epsilon_def}:
\begin{align}
	\dot W(t)
	=
	-h_\epsilon(\nabla f(W(t))),
	\qquad
	W(0)=W_0.
	\label{eq:direct_spectral_gradient_flow}
\end{align}
Throughout the analysis, we impose the following assumptions.

\begin{assumption}
	\label{ass:local_gradient_lipschitz}
The function $f:\mathbb R^{m\times n}\to\mathbb R$ is continuously
differentiable, and $\nabla f$ is Lipschitz continuous on every bounded
	subset of $\mathbb R^{m\times n}$.
\end{assumption}

\begin{assumption}
	\label{ass:lower_coercive}
The function $f$ is bounded below and coercive. In particular, there exists
$f_{\inf}\in\mathbb R$ such that $f(W)\ge f_{\inf}$ for every $W$.
\end{assumption}

\begin{proposition}[Global well-posedness of the direct flow]
	\label{prop:direct_flow_wellposedness}
Under Assumption~\ref{ass:local_gradient_lipschitz}, for every $W_0$,
system \eqref{eq:direct_spectral_gradient_flow} admits a unique global
	solution. Moreover,
	\[
	\|\dot W(t)\|_F\le\sqrt d,
	\qquad
	\|W(t)-W_0\|_F\le\sqrt d\,t.
	\]
\end{proposition}

The proof is provided in Appendix~\ref{app:direct_wellposedness}.

Let $G(t)=\nabla f(W(t))$. The key dissipation identity is
\begin{align}
	\begin{aligned}
	\frac{d}{dt}f(W(t))
	&=
	-\langle G(t),h_\epsilon(G(t))\rangle=
	-\sum_i
	\frac{\gamma_i(t)^2}
	{\sqrt{\gamma_i(t)^2+\epsilon}},
	\end{aligned}
	\label{eq:direct_flow_energy_dissipation}
\end{align}
where $\gamma_i(t)$ are the singular values of $G(t)$. Hence,
$f(W(t))$ is nonincreasing.

Define the stationary set \(\mathcal S=\{W\in\mathbb R^{m\times n}:\nabla f(W)=0\}\).

\begin{theorem}[Direct spectral flow convergence]
	\label{thm:direct_flow_convergence}
	Suppose that Assumptions~\ref{ass:local_gradient_lipschitz} and
\ref{ass:lower_coercive} hold. Then every solution of
	\eqref{eq:direct_spectral_gradient_flow} is bounded and satisfies
	\(\int_0^\infty\|\nabla f(W(t))\|_F^2\,dt<\infty\) and
	\(\|\nabla f(W(t))\|_F \to 0\). Consequently,
	\(\omega(W_0)\subseteq\mathcal S\).
\end{theorem}

The proof is provided in Appendix~\ref{app:direct_convergence}.

\subsection{Momentum-Augmented System}

The same spectral feedback law can be used in the momentum-augmented system
\begin{align}
	\begin{cases}
		\dot W=-h_\epsilon(M),\\[1mm]
		\dot M=a\nabla f(W)-bM,
	\end{cases}
	\qquad a>0,\quad b>0.
	\label{eq:momentum_augmented_flow}
\end{align}
Here $M(t)$ is an auxiliary momentum state. Define
\begin{align}
	V(W,M)
	=
	a\bigl(f(W)-f_{\inf}\bigr)+\Phi_\epsilon(M).
	\label{eq:momentum_storage}
\end{align}
Using $\nabla\Phi_\epsilon=h_\epsilon$, the cross terms cancel and yield
the exact identity
\begin{align}
	\dot V(W(t),M(t))
	=
	-b\langle h_\epsilon(M(t)),M(t)\rangle
	\le0.
	\label{eq:momentum_storage_dissipation}
\end{align}

\begin{theorem}[Momentum-augmented flow convergence]
	\label{thm:momentum_flow_convergence}
	Suppose that Assumptions~\ref{ass:local_gradient_lipschitz} and
\ref{ass:lower_coercive} hold. Every solution of
	\eqref{eq:momentum_augmented_flow} is global and bounded, and
	\(\omega(W_0,M_0)\subseteq\mathcal S\times\{0\}\).
	Equivalently, every limit point satisfies
	$\nabla f(W)=0$ and $M=0$.
\end{theorem}

The proof is provided in Appendix~\ref{app:momentum_convergence}.

\begin{remark}[Scope of global convergence]
Theorems~\ref{thm:direct_flow_convergence} and~\ref{thm:momentum_flow_convergence}
establish convergence to the stationary set. Without further structural assumptions,
such as isolated stationary points, they do not by themselves imply convergence to a
unique stationary point, nor do they quantify how rapidly the
trajectory approaches the stationary set. This distinction
motivates the quantitative estimates developed in the next
section.
\end{remark}

\subsection{Approximation of the Ideal Matrix-Polar Flow}
\label{sec:polar-flow-approximation}

The preceding analysis treats the smoothing parameter
$\epsilon>0$ as fixed. We now clarify in what sense the proposed
smoothed spectral feedback approximates the exact matrix-polar
feedback underlying Muon-type update directions.

Let $d=\min\{m,n\}$. For a full-column-rank matrix
$Z\in\mathbb{R}^{m\times n}$ with compact singular value
decomposition
\[
Z=U\operatorname{diag}(\sigma_1,\ldots,\sigma_d)V^\top,
    \qquad
    \sigma_{\min}(Z)>0,
\]
recall that the canonical matrix-polar factor is
$\operatorname{Polar}(Z)=UV^\top=Z(Z^\top Z)^{-1/2}.
$

The smoothed feedback can be written as
\[
    h_\epsilon(Z)
    =
    U\operatorname{diag}\left(
        \frac{\sigma_i}{\sqrt{\sigma_i^2+\epsilon}}
    \right)V^\top.
\]
Therefore, $h_\epsilon(Z)$ approaches $\operatorname{Polar}(Z)$ whenever
the singular values of $Z$ remain separated from zero.

To describe the corresponding trajectory-level approximation,
consider the ideal direct matrix-polar flow
\begin{equation}
    \dot W_{\mathrm P}(t)
    =
    -\operatorname{Polar}\bigl(\nabla f(W_{\mathrm P}(t))\bigr),
    \qquad
    W_{\mathrm P}(0)=W_0,
    \label{eq:ideal-direct-polar-flow}
\end{equation}
and the ideal momentum-augmented matrix-polar flow
\begin{equation}
    \begin{alignedat}{2}
        &\dot W_{\mathrm P}(t)
        &&=
        -\operatorname{Polar}(M_{\mathrm P}(t)),\\
        &\dot M_{\mathrm P}(t)
        &&=
        a\nabla f(W_{\mathrm P}(t))-bM_{\mathrm P}(t),\\
        &(W_{\mathrm P}(0),M_{\mathrm P}(0))
        &&=
        (W_0,M_0).
    \end{alignedat}
    \label{eq:ideal-momentum-polar-flow}
\end{equation}

\begin{proposition}[Finite-horizon approximation]
\label{prop:finite-horizon-polar-approximation}
Suppose that Assumption~\ref{ass:local_gradient_lipschitz} holds, and let
$\underline{\sigma}>0$.

First, for every full-rank matrix $Z$ satisfying
\[
    \sigma_{\min}(Z)\geq \underline{\sigma},
\]
the feedback approximation satisfies
\begin{equation}
    \left\|
        h_\epsilon(Z)-\operatorname{Polar}(Z)
    \right\|_F
    \leq
    \sqrt{d}
    \left(
        1-
        \frac{\underline{\sigma}}
        {\sqrt{\underline{\sigma}^2+\epsilon}}
    \right)
    \leq
    \frac{\sqrt{d}}{2\underline{\sigma}^2}\epsilon.
    \label{eq:uniform-feedback-approximation}
\end{equation}

Moreover, fix any $T>0$.

\begin{enumerate}
    \item Suppose that the ideal direct flow
    \eqref{eq:ideal-direct-polar-flow} has a classical solution
    on $[0,T]$ satisfying
    \begin{equation}
        \inf_{0\leq t\leq T}
        \sigma_{\min}\bigl(\nabla f(W_{\mathrm P}(t))\bigr)
        \geq 2\underline{\sigma}.
        \label{eq:direct-rank-safe-condition}
    \end{equation}
    Let $W_\epsilon(t)$ be the solution of the smoothed direct
    flow~\eqref{eq:direct_spectral_gradient_flow} with the same initial condition.
    Then there exist constants
    $\epsilon_T^{\mathrm{dir}}>0$ and
    $C_T^{\mathrm{dir}}>0$, independent of $\epsilon$, such that
    \begin{equation}
        \sup_{0\leq t\leq T}
        \left\|
            W_\epsilon(t)-W_{\mathrm P}(t)
        \right\|_F
        \leq
        C_T^{\mathrm{dir}}\epsilon,
        \qquad
        0<\epsilon\leq\epsilon_T^{\mathrm{dir}}.
        \label{eq:direct-trajectory-approximation}
    \end{equation}

    \item Suppose that the ideal momentum flow
    \eqref{eq:ideal-momentum-polar-flow} has a classical solution
    on $[0,T]$ satisfying
    \begin{equation}
        \inf_{0\leq t\leq T}
        \sigma_{\min}(M_{\mathrm P}(t))
        \geq 2\underline{\sigma}.
        \label{eq:momentum-rank-safe-condition}
    \end{equation}
    Let $(W_\epsilon(t),M_\epsilon(t))$ be the solution of the
    smoothed momentum system~\eqref{eq:momentum_augmented_flow} with the
    same initial condition. Then there exist constants
    $\epsilon_T^{\mathrm{mom}}>0$ and
    $C_T^{\mathrm{mom}}>0$, independent of $\epsilon$, such that
    \begin{equation}
    \begin{alignedat}{2}
        &E_{\mathrm{mom}}(t)
        &&:=
        \|W_\epsilon-W_{\mathrm P}\|_F
        +
        \|M_\epsilon-M_{\mathrm P}\|_F,\\
        &\sup_{0\leq t\leq T}E_{\mathrm{mom}}(t)
        &&\leq
        C_T^{\mathrm{mom}}\epsilon,
        \qquad
        0<\epsilon\leq\epsilon_T^{\mathrm{mom}}.
    \end{alignedat}
        \label{eq:momentum-trajectory-approximation}
    \end{equation}
\end{enumerate}

Thus, on every finite time interval over which the relevant
gradient or momentum singular values remain uniformly separated
from zero, the smoothed spectral flow converges uniformly to the
corresponding ideal matrix-polar flow as $\epsilon \to 0$.
\end{proposition}

The proof is provided in
Appendix~\ref{app:finite-horizon-polar-approximation}.

\begin{remark}[Scope of the matrix-polar approximation]
\label{rem:scope-polar-approximation}
Proposition~\ref{prop:finite-horizon-polar-approximation}
describes an ideal rank-safe regime in which the smoothed
spectral feedback reproduces the matrix-polar direction used by
Muon-type methods. It should not be interpreted as a complete
continuous-time characterization of the discrete Muon
algorithm. In particular, the proposition compares two
continuous-time vector fields on a fixed finite time interval;
it does not account for finite stepsizes, the discrete momentum
recursion, stochastic gradients, or numerical approximations of
the polar factor such as finite Newton--Schulz iterations.

The uniform lower singular-value assumption is essential. If the smallest relevant singular
value is allowed to approach zero, the constant in
\eqref{eq:uniform-feedback-approximation} degenerates, and the
convergence of $h_\epsilon$ to the matrix-polar map is no longer
uniform near rank changes. For example, in the scalar case,
taking $z_\epsilon=\epsilon>0$ gives
\[
    h_\epsilon(z_\epsilon)
    =
    \frac{\epsilon}{\sqrt{\epsilon^2+\epsilon}}
    =
    \sqrt{\frac{\epsilon}{1+\epsilon}}
    \longrightarrow 0,
\]
whereas the exact polar direction satisfies
$\operatorname{Polar}(z_\epsilon)=1$. Hence, although
$h_\epsilon(z)\to\operatorname{Polar}(z)$ for each fixed $z\neq0$, this
convergence is not uniform as $z$ approaches zero.

Consequently, the proposition provides a local-in-rank and
finite-horizon connection with ideal matrix-polar feedback, but
it does not capture the full dynamical behavior of Muon near
small singular values or rank transitions. Constructing a
continuous-time model that remains valid across such singular
regimes and more faithfully represents the complete Muon
dynamics is left for future work.
\end{remark}

\section{Convergence Rate Results}
\label{sec:trajectory_convergence}
The preceding results establish the stability and convergence structure of the direct and momentum-augmented spectral gradient flows. We now derive quantitative convergence rate results by converting the dissipation of \(h_\epsilon\) into integral gradient bounds, stationarity rates, objective-gap bounds, and averaged convergence certificates. For each system, the analysis begins with a common integral estimate and then specializes that estimate to different objective landscapes.

\subsection{Convergence Rate of the Spectral Gradient Flow}
Along the direct flow \eqref{eq:direct_spectral_gradient_flow}, define the initial sublevel set and constants
\[
\mathcal L_0=\{W:f(W)\le f(W_0)\},
\qquad
K_G=\sup_{W\in\mathcal L_0}\|\nabla f(W)\|_F,
\]
\[
c_\epsilon=\frac1{\sqrt{K_G^2+\epsilon}}.
\]
Coercivity makes $\mathcal L_0$ bounded. The dissipation identity implies
\begin{align}
	\frac{d}{dt}f(W(t))
	\le
	-c_\epsilon\|\nabla f(W(t))\|_F^2.
	\label{eq:direct_flow_dissipation_again}
\end{align}

\begin{lemma}[Direct flow trajectory bounds]
	\label{thm:direct_flow_trajectory_bounds}
	Suppose that Assumptions~\ref{ass:local_gradient_lipschitz} and
	\ref{ass:lower_coercive} hold. For every $T>0$, the solution $W(t)$ satisfies
	\[
	\int_0^T\|\nabla f(W(t))\|_F^2\,dt
	\le
	\frac{f(W_0)-f_{\inf}}{c_\epsilon},
	\]
	and
	\[
	\int_0^T\|\dot W(t)\|_F^2\,dt
	\le
	\frac{f(W_0)-f_{\inf}}{\sqrt{\epsilon}}.
	\]
	In particular, both integrals remain finite on $[0,\infty)$.
\end{lemma}

The proof is provided in Appendix~\ref{app:direct_trajectory_bounds}.

Lemma~\ref{thm:direct_flow_trajectory_bounds} provides the basic
dissipation budget for the direct flow. Without additional structure on
the objective, averaging the gradient-energy estimate over a finite time
horizon immediately yields a nonconvex stationarity certificate.

\begin{theorem}[Direct nonconvex stationarity rate]
	\label{thm:direct_flow_nonconvex_rate}
	Under the assumptions of Lemma~\ref{thm:direct_flow_trajectory_bounds},
	\[
	\min_{0\le t\le T}\|\nabla f(W(t))\|_F^2
	\le
	\frac{f(W_0)-f_{\inf}}{c_\epsilon T},
	\qquad T>0.
	\]
	Equivalently, the minimum squared gradient norm is
	$\mathcal O(1/T)$. Moreover, $\|\nabla f(W(t))\|_F\to0$.
\end{theorem}

The proof is provided in Appendix~\ref{app:direct_nonconvex}.

The preceding estimate controls the minimum gradient norm along the
trajectory. Under convexity, the gradient also controls the objective gap
through the distance to the solution set, which converts the same
dissipation inequality into a pointwise objective estimate. Assume now
that $f$ is convex,
\(\mathcal W^*=\operatorname*{arg\,min}_{W}f(W)\ne\emptyset\), and
\(f_*=\min_W f(W)\). Define
\(R_0=\sup_{W\in\mathcal L_0}\operatorname{dist}(W,\mathcal W^*)\).

\begin{theorem}[Direct convex convergence rate]
	\label{thm:direct_flow_convex_rate}
	Suppose that Assumptions~\ref{ass:local_gradient_lipschitz} and
	\ref{ass:lower_coercive} hold and that $f$ is convex. If
	$f(W_0)>f_*$, then
	\[
	f(W(t))-f_*
	\le
	\frac{R_0^2(f(W_0)-f_*)}
	{R_0^2+c_\epsilon t(f(W_0)-f_*)},
	\qquad t\ge0.
	\]
	In particular, $f(W(t))-f_*=\mathcal O(1/t)$.
\end{theorem}

The proof is provided in Appendix~\ref{app:direct_convex}.

Convexity gives a sublinear objective-gap bound. A PL inequality provides
a stronger gradient-dominance relation, so the dissipation estimate
closes directly as a linear differential inequality in the objective
gap and yields exponential decay.

\begin{theorem}[Direct PL exponential convergence]
	\label{thm:direct_flow_pl_rate}
	Suppose that Assumptions~\ref{ass:local_gradient_lipschitz} and
	\ref{ass:lower_coercive} hold and that, on $\mathcal L_0$,
	\[
	\frac12\|\nabla f(W)\|_F^2
	\ge
	\mu(f(W)-f_*).
	\]
	Then
	\[
	f(W(t))-f_*
	\le
	e^{-2\mu c_\epsilon t}(f(W_0)-f_*),
	\qquad t\ge0.
	\]
	In particular, the estimate applies when $f$ is
	$\mu$-strongly convex.
\end{theorem}

The proof is provided in Appendix~\ref{app:direct_pl}.

For the direct flow, all three conclusions follow from the fact that the
gradient signal appears explicitly in the dissipation term. The
momentum-augmented system requires an additional transfer step because
the Lyapunov identity directly dissipates the momentum state instead.

\subsection{Convergence Rate of the Momentum-Augmented Flow}

We first quantify the energy available in the momentum channel. Let
\(V_0=V(W_0,M_0)\) and
\(K_M=\sup_{t\ge0}\|M(t)\|_F\).

\begin{lemma}[Momentum-augmented trajectory bounds]
	\label{thm:momentum_flow_trajectory_bounds}
	Under Assumptions~\ref{ass:local_gradient_lipschitz} and
	\ref{ass:lower_coercive}, every solution $(W(t),M(t))$ of
	\eqref{eq:momentum_augmented_flow} satisfies
	\[
	\int_0^\infty\|M(t)\|_F^2\,dt
	\le
	\frac{\sqrt{K_M^2+\epsilon}}{b}V_0,
	\]
	and
	\[
	\int_0^\infty\|\dot W(t)\|_F^2\,dt
	\le
	\frac{V_0}{b\sqrt{\epsilon}}.
	\]
\end{lemma}

The proof is provided in Appendix~\ref{app:momentum_trajectory_bounds}.

The momentum estimate alone does not directly bound the gradient signal.
The cross quantity below couples the gradient and momentum states and
transfers the integral control of \(M\) to a stationarity estimate.
Define \(Q(W,M)=\langle\nabla f(W),M\rangle\) and
\(\Delta_Q=\sup_{t\ge0}Q(W(t),M(t))-Q(W_0,M_0)\). On a bounded set
\(\mathcal K\) containing the trajectory, let
\(C_M=b^2/(2a)+L_{\mathcal K}/\sqrt{\epsilon}\).

\begin{theorem}[Momentum-augmented stationarity rate]
	\label{thm:momentum_flow_nonconvex_rate}
	Under Assumptions~\ref{ass:local_gradient_lipschitz} and
	\ref{ass:lower_coercive}, for all $T>0$, we have
	\[
	\begin{aligned}
	\min_{0\le t\le T}\|\nabla f(W(t))\|_F^2
	&\le
	\frac{2\Delta_Q}{aT}+
	\frac{2C_M}{aT}
	\frac{\sqrt{K_M^2+\epsilon}}{b}V_0.
	\end{aligned}
	\]
\end{theorem}

The proof is provided in Appendix~\ref{app:momentum_nonconvex}.

For a convex objective, the momentum equation also relates the integrated
gradient signal to the accumulated objective gap. Combining this relation
with the momentum-energy bound and Jensen's inequality leads naturally to
an averaged state. Define the Ces\`aro average
\(\bar W_T=T^{-1}\int_0^T W(t)\,dt\).

\begin{theorem}[Momentum-augmented convex average rate]
	\label{thm:momentum_flow_convex_rate}
	Suppose that Assumptions~\ref{ass:local_gradient_lipschitz} and
	\ref{ass:lower_coercive} hold, $f$ is convex, and
	$\mathcal W^*\ne\emptyset$. There exists $C_{\rm mom}>0$, independent
	of $T$, such that
	\[
	f(\bar W_T)-f_*
	\le
	\frac{C_{\rm mom}}{\sqrt T},
	\qquad T\ge1.
	\]
	Equivalently, $f(\bar W_T)-f_*=\mathcal O(T^{-1/2})$.
\end{theorem}

The proof is provided in Appendix~\ref{app:momentum_convex}. These
results show how momentum-storage dissipation is transferred to
velocity, stationarity, and convergence-rate estimates.

\begin{remark}[Interpretation of global rate bounds]
\label{rem:global-rate-interpretation}
The bounds in this section are obtained through
trajectory-wise norm bounds and sector inequalities for
$h_\epsilon$. In this process, the detailed spectral information
of the gradient or momentum matrix is compressed into scalar
quantities such as $K_G$, $K_M$, and $\epsilon$. In particular,
the bounds do not retain the singular-value distribution, the
singular-vector orientation, or their interaction with the local
curvature of the objective. Consequently, the resulting
asymptotic orders alone do not exhibit a systematic global
advantage of the spectral flow over the standard
Frobenius-gradient flow. They should therefore be interpreted
primarily as stability and convergence certificates, rather than
as a global complexity comparison between the two directions.

This observation suggests that any benefit of spectral
normalization is unlikely to be universal. Instead, it may arise
only in regions where the singular-value structure and singular-vector
alignment of the gradient interact favorably with the directional
curvature induced by the Hessian. To expose this
mechanism, the next section moves beyond global norm estimates
and compares the two update directions through a local model.
\end{remark}

\section{Local Descent-Rate Comparison}
\label{sec:local_descent_comparison}

\subsection{General Local Advantage Criterion}
\label{subsec:general_local_comparison}
The global results do not explain why replacing Frobenius-gradient
feedback with spectral gradient feedback may improve optimization performance.
To better understand this mechanism, we study
the optimization process locally. Specifically, at a fixed point \(W\), we
compare the objective descent rates generated by the
Frobenius-gradient flow and the spectral gradient flow. This
comparison provides a local criterion for identifying when spectral
normalization yields a larger instantaneous decrease of the objective.

Let $f$ be $C^2$ near $W$. We denote
$G=\nabla f(W)$ and the Hessian operator by
\[
\mathcal H_W[D]
=
\left.
\frac{d}{ds}\nabla f(W+sD)
\right|_{s=0}.
\]

\begin{assumption}[Admissible comparison directions]
\label{ass:local_direction}

Let \(D\in\mathbb R^{m\times n}\) be a nonzero direction at \(W\).
We assume that
\[
\langle G,D\rangle>0,
\qquad
\langle D,\mathcal H_W[D]\rangle>0.
\]

\end{assumption}

A direct comparison of the two unscaled flows, $\dot{W}=-G$ and
$\dot{W}=-h_{\epsilon}(G)$, is affected by the arbitrary
choice of time scale.
To compare their intrinsic local descent
efficiencies, we assign each direction a direction-dependent factor $\alpha$ that maximizes the decrease predicted by the same local
quadratic model. This normalization removes the arbitrary scaling of the update
directions and isolates their interaction with the directional curvature
induced by the local Hessian.

For an update direction \(D\), the second-order expansion along
\(W-\alpha D\) is
\[
f(W-\alpha D)
=
f(W)
-\alpha\langle G,D\rangle
+\frac{\alpha^2}{2}
\langle D,\mathcal H_W[D]\rangle
+o(\alpha^2).
\]
Ignoring the higher-order term, we define the local quadratic model
\(q_D(\alpha)=f(W)-\alpha\langle G,D\rangle+\frac{\alpha^2}{2}
\langle D,\mathcal H_W[D]\rangle\).

Under Assumption~\ref{ass:local_direction}, \(q_D\) is strictly convex
and the unique minimizer of \(q_D\) over \(\alpha>0\) is
\[\alpha_D^\ast=\frac{\langle G,D\rangle}{\langle D,\mathcal H_W[D]\rangle}.\]
We then use \(\alpha_D^\ast\) to calibrate the local time scale of the
continuous flow and consider
\[
    \dot W=-\alpha_D^\ast D.
\]
This choice has a direct finite-step interpretation. The
unit-time Euler step associated with this calibrated flow attains the
largest objective decrease predicted by the quadratic model.

Under the local model, the resulting objective descent rate at \(W\) is
\[
R_W(D)
:=
-\frac{d}{dt}f(W(t))
=
-\langle G,\dot W_D\rangle
=
\frac{\langle G,D\rangle^2}
{\langle D,\mathcal H_W[D]\rangle}.
\]

For the Frobenius gradient flow $\dot{W}=-G$, the update direction is
$D=G$; for the smoothed spectral gradient flow
$\dot{W}=-h_{\epsilon}(G)$, the update direction is
$D=h_{\epsilon}(G)$. Substituting
these directions into the same time-scale calibration gives,
respectively,
\begin{align}
	R_F(W)
	=
	\frac{\|G\|_F^4}
	{\langle G,\mathcal H_W[G]\rangle},
	\label{eq:RF_definition}
\end{align}
and
\begin{align}
	R_\epsilon(W)
	=
	\frac{\langle G,h_\epsilon(G)\rangle^2}
	{\left\langle
	h_\epsilon(G),\mathcal H_W[h_\epsilon(G)]
	\right\rangle}.
	\label{eq:Repsilon_definition}
\end{align}

The local descent efficiencies of the two gradient flows are compared
through $R_F(W)$ and $R_{\epsilon}(W)$. If
$R_F(W)>R_{\epsilon}(W)$, the Frobenius gradient flow descends faster;
the reverse inequality favors the smoothed spectral gradient flow.

Define the local advantage ratio
\begin{align}
	\Gamma_\epsilon(W)
	=
	\frac{R_\epsilon(W)}{R_F(W)}.
	\label{eq:Gamma_epsilon_definition}
\end{align}
The comparison threshold for the two local descent efficiencies is one.
This observation gives the following result.

\begin{theorem}[General local spectral advantage criterion]
\label{thm:local_advantage_condition}
Let \(f\) be \(C^2\) near \(W\), let \(G=\nabla f(W)\neq0\), and assume
that the directions \(G\) and \(h_\epsilon(G)\) satisfy
Assumption~\ref{ass:local_direction}. Then the local
descent-rate ratio is
\[
\Gamma_\epsilon(W)
=
\frac{
\langle G,h_\epsilon(G)\rangle^2
\langle G,\mathcal H_W[G]\rangle
}{
\|G\|_F^4
\left\langle
h_\epsilon(G),\mathcal H_W[h_\epsilon(G)]
\right\rangle
}.
\]
Under the calibrated time scales, the smoothed
spectral-gradient direction has a strictly larger instantaneous objective
descent rate than the Frobenius-gradient direction if and only if
\(\Gamma_\epsilon(W)>1\).
\end{theorem}

The proof is provided in Appendix~\ref{app:local_advantage}.

\begin{remark}[Interpretation of the criterion]
Theorem~\ref{thm:local_advantage_condition} provides a local comparison
between the Frobenius-gradient and smoothed spectral-gradient directions
under a common time-scale calibration. The quantity
\(\Gamma_\epsilon(W)\) is the ratio of the resulting instantaneous
objective descent rates. Since it depends on both the gradient and the
local Hessian, the local advantage of the spectral direction is jointly
determined by first-order alignment and directional curvature. For
objectives with additional structure, the criterion can be simplified
further. The quadratic regression case is analyzed in
Subsection~\ref{subsec:quadratic_mlp}.
\end{remark}

\subsection{Quadratic Loss for a Neural-Network Layer}
\label{subsec:quadratic_mlp}

Many optimization problems arising in neural-network training involve
matrix-valued parameters and data representations. We consider one of the
simplest such problems: a linear regression problem associated with a
single network layer.
Consider a weight matrix
$W\in\mathbb{R}^{m\times n}$. Let
$A\in\mathbb{R}^{n\times N}$ denote the current matrix of input
features or activations for this layer, and let \(Y\) be the desired
output. Conditional on the
remaining network variables, consider the layerwise squared-loss
model
\begin{equation}
    f(W)
    =
    \frac{1}{2}\|WA-Y\|_F^2,
    \label{eq:fixed_activation_quadratic_loss}
\end{equation}
where $Y\in\mathbb{R}^{m\times N}$ is the corresponding target
matrix.

The matrix $A$ is treated as given when differentiating with
respect to $W$. Hence, the analysis below is
pointwise: at each training state, the criterion may be evaluated
using the current activation matrix.

For \eqref{eq:fixed_activation_quadratic_loss}, the gradient is
\begin{equation}
    G
    =
    \nabla f(W)
    =
    (WA-Y)A^\top,
    \label{eq:quadratic_gradient}
\end{equation}
and the Hessian action along a direction \(D\) is
\begin{equation}
    \mathcal H_W[D]
    =
    DAA^\top.
    \label{eq:quadratic_hessian}
\end{equation}
In particular, the Hessian is independent of \(W\). Moreover,
\begin{equation}
    \langle D,\mathcal H_W[D]\rangle
    =
    \operatorname{Tr}(D^\top DAA^\top)
    =
    \|DA\|_F^2.
    \label{eq:quadratic_directional_curvature}
\end{equation}
Therefore, the general local rate associated with an admissible direction
\(D\) reduces to
\begin{equation}
    R_W(D)
    =
    \frac{\langle G,D\rangle^2}
    {\|DA\|_F^2}.
    \label{eq:quadratic_general_rate}
\end{equation}

Taking \(D=G\) and \(D=h_\epsilon(G)\), respectively, gives
\begin{align}
    R_F(W)
    &=
    \frac{\|G\|_F^4}
    {\langle G^\top G,AA^\top\rangle},
    \label{eq:quadratic_RF}
    \\
    R_\epsilon(W)
    &=
    \frac{\langle G,h_\epsilon(G)\rangle^2}
    {\left\langle
    h_\epsilon(G)^\top h_\epsilon(G),AA^\top
    \right\rangle}.
    \label{eq:quadratic_Repsilon}
\end{align}
These expressions separate the contribution of the gradient matrix from
the curvature induced by the feature matrix.

To obtain a particularly transparent criterion, we consider the square
full-rank case and the matrix-polar limit. Let
\(W,A,Y,G\in\mathbb R^{d\times d}\) and assume that \(G\) is nonsingular.

As \(\epsilon\to 0\),
\[
h_\epsilon(G)\longrightarrow \operatorname{Polar}(G).
\]

Moreover,
\[
\langle G,h_\epsilon(G)\rangle
\longrightarrow
\langle G,\operatorname{Polar}(G)\rangle
=
\|G\|_*,
\]
and, since
\(\operatorname{Polar}(G)^\top\operatorname{Polar}(G)=I\),
\[
\begin{aligned}
\left\langle
h_\epsilon(G),
\mathcal H_W[h_\epsilon(G)]
\right\rangle
&=
\left\|h_\epsilon(G)A\right\|_F^2
\\
&\longrightarrow
\|\operatorname{Polar}(G)A\|_F^2
=
\|A\|_F^2.
\end{aligned}
\]
Therefore, the limiting matrix-polar rate is
\begin{equation}
    R_0(W)
    :=
    \lim_{\epsilon \to 0}R_\epsilon(W)
    =
    \frac{\|G\|_*^2}{\|A\|_F^2}.
    \label{eq:quadratic_R0}
\end{equation}
Combining this expression with
\begin{equation}
    R_F(W)
    =
    \frac{\|G\|_F^4}
    {\langle G^\top G,AA^\top\rangle},
    \label{eq:quadratic_RF_limit}
\end{equation}
gives
\begin{equation}
\begin{aligned}
    \Gamma_0(W)
    &:=
    \lim_{\epsilon \to 0}\Gamma_\epsilon(W)
    =
    \frac{R_0(W)}{R_F(W)}
    \\
    &=
    \frac{\|G\|_*^2}{\|G\|_F^2}
    \frac{\langle G^\top G,AA^\top\rangle}
    {\|G\|_F^2\|A\|_F^2}.
\end{aligned}
\label{eq:quadratic_gamma_limit}
\end{equation}

We introduce the effective rank of the gradient matrix
\begin{equation}
    r_{\mathrm{eff}}(G)
    :=
    \frac{\|G\|_*^2}{\|G\|_F^2}
    =
    \frac{\left(\sum_{i=1}^{d}\sigma_i(G)\right)^2}
    {\sum_{i=1}^{d}\sigma_i(G)^2}.
    \label{eq:gradient_effective_rank}
\end{equation}
We also define the normalized gradient and activation covariance
matrices
\begin{equation}
    \rho_G
    :=
    \frac{G^\top G}{\|G\|_F^2},
    \qquad
    \rho_A
    :=
    \frac{AA^\top}{\|A\|_F^2}.
    \label{eq:normalized_covariances}
\end{equation}
We measure the similarity between $\rho_G$ and $\rho_A$ using the
Frobenius inner product:
\begin{equation}
    \operatorname{sim}_{\mathrm{HS}}(\rho_G,\rho_A)
    :=
    \langle\rho_G,\rho_A\rangle
    =
    \operatorname{Tr}(\rho_G^\top\rho_A).
    \label{eq:HS_covariance_similarity}
\end{equation}
With these definitions, the limiting advantage ratio becomes
\begin{equation}
    \Gamma_0(W)
    =
    r_{\mathrm{eff}}(G)
    \operatorname{sim}_{\mathrm{HS}}(\rho_G,\rho_A).
    \label{eq:quadratic_gamma_factorization}
\end{equation}

\begin{corollary}[Advantage criterion for the quadratic loss]
\label{cor:quadratic_local_advantage}
Consider the layerwise quadratic model
\eqref{eq:fixed_activation_quadratic_loss}. Suppose that
\(W,A,Y\in\mathbb R^{d\times d}\), that
\(G=\nabla f(W)\) is nonsingular, and that the Frobenius-gradient and
matrix-polar directions satisfy
Assumption~\ref{ass:local_direction}. Then the matrix-polar spectral
gradient flow has a strictly larger locally optimal descent rate than the
Frobenius-gradient flow if and only if
\begin{align}
    r_{\mathrm{eff}}(G)
    \operatorname{sim}_{\mathrm{HS}}(\rho_G,\rho_A)
    >
    1.
    \label{eq:quadratic_advantage_condition}
\end{align}
Equivalently,
\begin{align}
    \operatorname{sim}_{\mathrm{HS}}(\rho_G,\rho_A)
    >
    \frac{1}{r_{\mathrm{eff}}(G)}.
    \label{eq:quadratic_similarity_threshold}
\end{align}
\end{corollary}

\begin{remark}[Interpretation of the quadratic criterion]
\label{rem:quadratic_criterion_interpretation}

The effective rank
\(r_{\mathrm{eff}}(G)\)
reflects the degree of isotropy in the singular-value spectrum of the gradient
matrix. A larger effective rank indicates that the singular values of
\(G\) are more evenly distributed, whereas a smaller effective rank
indicates that the gradient is dominated by a small number of singular
modes.

\(\operatorname{sim}_{\mathrm{HS}}(\rho_G,\rho_A)\)
measures the similarity between the normalized gradient covariance
\(G^\top G\) and the normalized activation covariance \(AA^\top\).
A larger value means that the gradient and activation matrices distribute
their energy along similar directions in the activation space, whereas a
smaller value indicates that their main energy directions are different.

Corollary~\ref{cor:quadratic_local_advantage} therefore shows that the
local advantage of the matrix-polar spectral gradient flow is determined
by two factors: the uniformity of the gradient singular-value spectrum
and the similarity between the gradient and activation covariance
structures. In particular, the spectral gradient flow is more likely to
achieve a larger local descent rate when the singular values of \(G\) are
relatively well distributed and the gradient covariance is sufficiently
similar to the activation covariance. 

The numerical experiments later vary the spectral anisotropy and relative
orientation of the activation matrix to illustrate how these two factors
affect the local advantage ratio.
\end{remark}

\section{Numerical Experiments}
\label{sec:numerical_experiments}

We present three numerical experiments that illustrate the theoretical
results developed above. The first assesses the modeling relevance of
the proposed continuous-time systems by comparing their trajectories with
those of discrete exact-polar spectral algorithms. The second isolates the effects of gradient effective rank
and gradient--activation covariance similarity in the simplified
neural-network model. The third compares the observed trajectory decay with the
convergence bounds derived for the smoothed spectral gradient flow.
All matrix functions are evaluated with NumPy, and continuous-time
trajectories are integrated using the classical RK4
method. The detailed constructions are included in the corresponding
subsections below.

\subsection{Comparison with Discrete Spectral Algorithms}
\label{subsec:discrete_continuous_comparison}

We first assess whether the proposed smoothed continuous-time models
capture the trajectories of the corresponding discrete exact-polar
algorithms. We consider
\begin{equation}
    f(W)=\frac12\|W\|_F^2,
    \qquad
    \nabla f(W)=W,
    \label{eq:discrete_continuous_quadratic}
\end{equation}
with
\[
    W_0
    =
    U\operatorname{diag}(4.13,3.07,2.03,1.237)V^\top,
\]
where \(U,V\in\mathbb R^{4\times4}\) are randomly generated orthogonal
matrices. The initial gradient is therefore full rank and separated from
zero, while its singular values necessarily become small as the
trajectories approach the minimizer \(W^\star=0\).

For the direct spectral method, we compare
\begin{align}
    W_{k+1}^{\mathrm d}
    &=
    W_k^{\mathrm d}
    -
    \eta\operatorname{Polar}\bigl(\nabla f(W_k^{\mathrm d})\bigr),
    \label{eq:discrete_direct_polar_method}
    \\
    \dot W_\epsilon(t)
    &=
    -h_\epsilon\bigl(\nabla f(W_\epsilon(t))\bigr).
    \label{eq:continuous_direct_smoothed_model}
\end{align}
For the momentum case, we compare the normalized exact-polar Muon-type
iteration
\begin{align}
    M_{k+1}^{\mathrm d}
    &=
    \beta M_k^{\mathrm d}
    +(1-\beta)\nabla f(W_k^{\mathrm d}),
    \nonumber\\
    W_{k+1}^{\mathrm d}
    &=
    W_k^{\mathrm d}
    -\eta\operatorname{Polar}(M_{k+1}^{\mathrm d}),
    \label{eq:discrete_exact_polar_muon}
\end{align}
with the smoothed momentum model
\begin{align}
    \dot W_\epsilon(t)
    &=
    -h_\epsilon(M_\epsilon(t)),
    \nonumber\\
    \dot M_\epsilon(t)
    &=
    \lambda
    \bigl(
        \nabla f(W_\epsilon(t))-M_\epsilon(t)
    \bigr).
    \label{eq:continuous_smoothed_momentum_model}
\end{align}
We initialize \(M_0=\nabla f(W_0)\) and choose
\(\beta=e^{-\lambda\eta}\) to match the momentum-decay time scale.

The continuous trajectories are sampled at \(t_k=k\eta\). We report the
cumulative discrepancies
\begin{align}
    \mathcal E_{\mathrm{dir}}(t_k)
    &=
    \max_{0\leq j\leq k}
    \|W_\epsilon(t_j)-W_j^{\mathrm d}\|_F,
    \label{eq:direct_discrete_continuous_error}
    \\
    \mathcal E_{\mathrm{mom}}(t_k)
    &=
    \max_{0\leq j\leq k}
    \left(
        \|W_\epsilon(t_j)-W_j^{\mathrm d}\|_F
        +
        \|M_\epsilon(t_j)-M_j^{\mathrm d}\|_F
    \right).
    \label{eq:momentum_discrete_continuous_error}
\end{align}
We set
\(
    \eta=2\times10^{-2},
    \qquad
    \lambda=2,
    \qquad
    \beta=e^{-0.04},
\)
and compare the smoothing parameters
\(
    \epsilon\in\{10^{-3},10^{-4},10^{-5}\}.
\)
The continuous systems are integrated
by RK4 with multiple substeps within each discrete interval.

\begin{figure*}[t]
    \centering
    \includegraphics[width=0.82\textwidth]
    {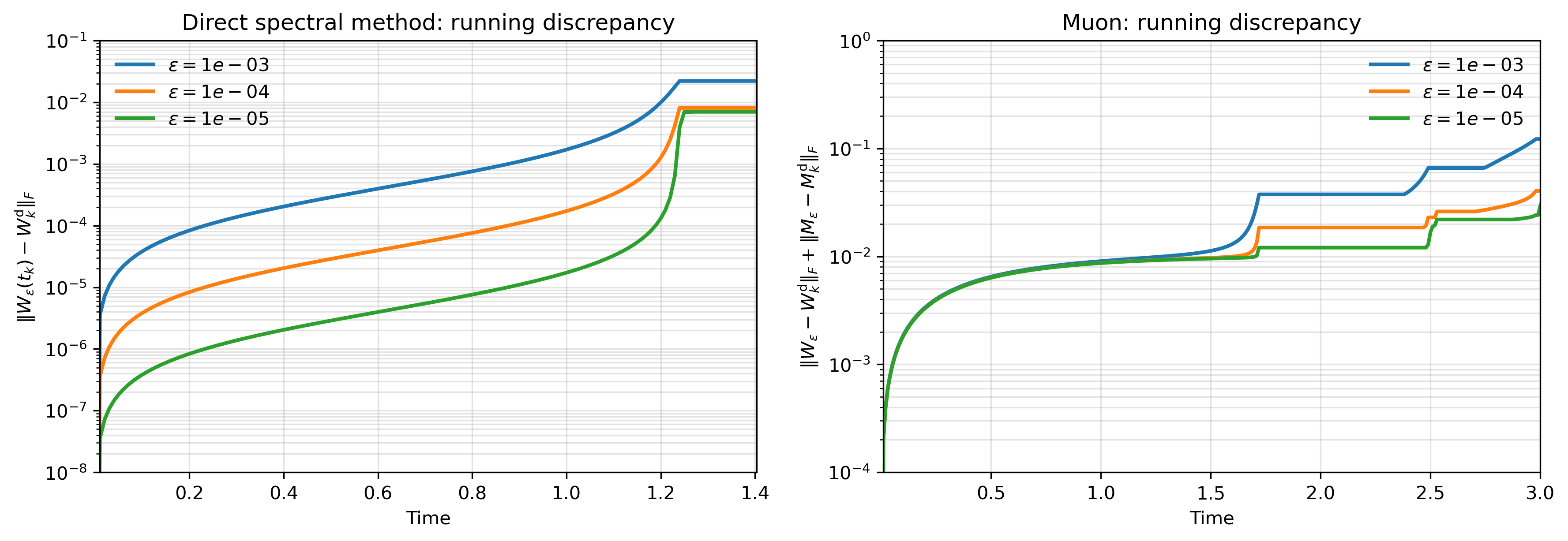}
    \caption{
    Cumulative trajectory discrepancies between the discrete exact-polar
    algorithms and the corresponding smoothed continuous-time models.
    Left: direct spectral-gradient method.
    Right: momentum-augmented Muon-type method.
    }
    \label{fig:discrete_continuous_discrepancy}
\end{figure*}

Figure~\ref{fig:discrete_continuous_discrepancy} shows that the
continuous models remain close to the corresponding discrete trajectories
during the early stage of optimization, with particularly small error in
the direct case. The momentum discrepancy is larger because it also
contains the finite-step approximation error of the momentum recursion.
Both the direct and momentum experiments exhibit the same overall ordering: smaller
values of \(\epsilon\) generally lead to smaller discrepancies.

As the trajectories approach \(W^\star=0\), the cumulative errors increase.
This behavior is consistent with the loss of the rank-safe regime: the
exact polar map maintains unit response on nonzero singular modes, whereas
the smoothed factor
\(
    \frac{\sigma}{\sqrt{\sigma^2+\epsilon}}
\)
vanishes continuously as \(\sigma\to0\). The experiment therefore
supports the proposed flows as finite-time models of discrete spectral
algorithms away from rank degeneracy, but not as globally exact
representations near zero singular values or rank transitions.

\subsection{Effects of Gradient Effective Rank and Covariance Similarity}
\label{subsec:numerical_local_advantage}

We next examine the fixed-activation quadratic model
 \(f(W)=\frac12\|WA-Y\|_F^2\) and the criterion
\begin{equation}
    \Gamma(W_0)
    =
    r_{\mathrm{eff}}(G_0)
    \operatorname{sim}_{\mathrm{HS}}(\rho_{G_0},\rho_A),
    \qquad
    G_0=\nabla f(W_0).
    \label{eq:numerical_gamma}
\end{equation}

Two controlled sweeps are performed. First,
\(r_{\mathrm{eff}}(G_0)=1.6\) is fixed while the covariance similarity is
varied, giving the theoretical transition \(\operatorname{sim}_{\mathrm{HS}}(\rho_{G_0},\rho_A)=1/1.6=0.625\).
Second, the covariance similarity is fixed at \(0.7\), while the gradient
effective rank is varied, giving the transition \(r_{\mathrm{eff}}(G_0)=1/0.7=10/7\).
Both sweeps compare the Frobenius direction with the exact matrix-polar
direction \(\operatorname{Polar}(G)\), using the local
normalizations
\begin{align*}
    \alpha_F(W)&=\frac{\|G(W)\|_F^2}{\|G(W)A\|_F^2},\\
    \alpha_P(W)&=\frac{\langle G(W),\operatorname{Polar}(G(W))\rangle}
    {\|\operatorname{Polar}(G(W))A\|_F^2}.
\end{align*}

For the covariance-similarity sweep, we set \(W_0=0\),
\(G_0=\operatorname{diag}(\sqrt{0.9},\sqrt{0.1})\), and
\begin{align*}
    A(\theta)
    &=R(\theta)\operatorname{diag}(\sqrt{0.95},\sqrt{0.05}),\\
    Y(\theta)&=-G_0A(\theta)^{-\top},
    \qquad \theta\in[0,\pi/2].
\end{align*}
This construction gives \(r_{\mathrm{eff}}(G_0)=1.6\) and preserves
\(\nabla f(W_0)=G_0\). We use \(151\) equally spaced values of \(\theta\)
and estimate the initial slopes with \(\Delta t=10^{-5}\). The
representative trajectories in
Figure~\ref{fig:representative_loss_curves} use \(\theta=0\) and
\(\theta=\pi/4\), yielding \(\Gamma(W_0)=1.376\) and \(0.8\),
respectively; they are integrated over \([0,0.08]\) with a time step of
\(10^{-4}\).

For the effective-rank sweep, let \(s_0=0.7\) and
\begin{align*}
    G_0(p)&=\operatorname{diag}(\sqrt p,\sqrt{1-p}),\\
    A(p)&=\operatorname{diag}(\sqrt{q(p)},\sqrt{1-q(p)}),\\
    q(p)&=\frac{s_0-1+p}{2p-1},
    \qquad p\in[0.72,0.99].
\end{align*}
Then \(\|G_0(p)\|_F=1\),
\(r_{\mathrm{eff}}(G_0(p))=(\sqrt p+\sqrt{1-p})^2\), and the construction
fixes \(\operatorname{sim}_{\mathrm{HS}}(\rho_{G_0(p)},\rho_{A(p)})=s_0\).
Setting \(W_0=0\) and \(Y(p)=-G_0(p)A(p)^{-\top}\) enforces
\(\nabla f(W_0)=G_0(p)\). The sweep uses \(151\) values of \(p\) and
\(\Delta t=10^{-5}\).

\begin{figure*}[t]
    \centering
    \includegraphics[width=0.78\textwidth]
    {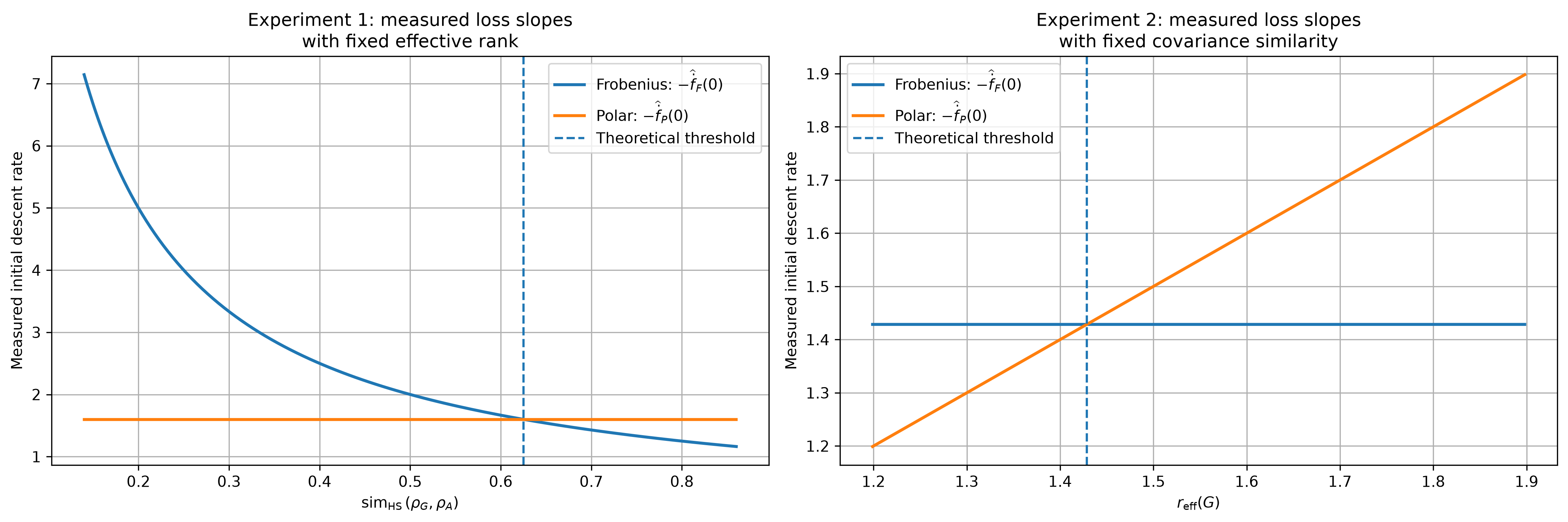}
    \caption{
    Measured initial descent rates under controlled variations of
    covariance similarity and gradient effective rank. The crossings occur
    near the theoretical thresholds determined by \(\Gamma(W_0)=1\).
    }
    \label{fig:local_rate_sweeps}
\end{figure*}

Figure~\ref{fig:local_rate_sweeps} shows that the ordering of the two
measured descent rates changes near the predicted thresholds in both
experiments. The matrix-polar direction produces a larger descent rate when
\(\Gamma(W_0)>1\), while the Frobenius-gradient direction produces a larger
descent rate when \(\Gamma(W_0)<1\).

To visualize this comparison, we select two representative
covariance-similarity settings. They satisfy
\(\Gamma(W_0)=1.376>1\) and \(\Gamma(W_0)=0.8<1\), respectively.

\begin{figure*}[t]
    \centering
    \includegraphics[width=0.78\textwidth]
    {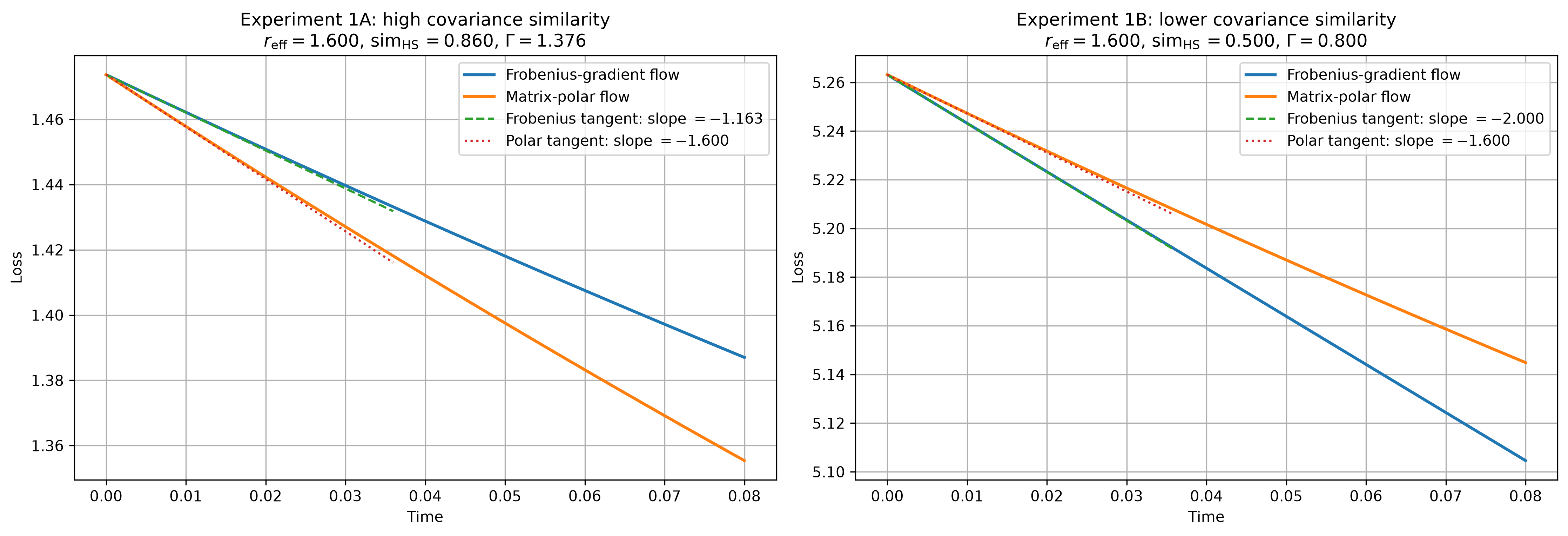}
    \caption{
    Representative loss trajectories and initial tangent lines.
    The matrix-polar trajectory has the steeper initial decrease when
    \(\Gamma(W_0)>1\), whereas the Frobenius-gradient trajectory is
    steeper when \(\Gamma(W_0)<1\).
    }
    \label{fig:representative_loss_curves}
\end{figure*}

The results illustrate that the local matrix-polar advantage is jointly
controlled by the singular-value distribution of the gradient and the
covariance alignment between the gradient and activation matrices.

\subsection{Verification of the Trajectory Convergence Bounds}
\label{subsec:trajectory_bound_verification}

Finally, we examine the convergence bounds for the direct smoothed
spectral gradient flow
\begin{align}
    \dot W(t)=-h_\epsilon\bigl(\nabla f(W(t))\bigr).
    \label{eq:numerical_direct_spectral_flow}
\end{align}
Let \(E(t)=f(W(t))-f^\star\).
The flow is integrated without the local scalar normalization used in the
slope comparisons. In both tests, \(W\in\mathbb R^{12\times12}\),
\(\epsilon=10^{-2}\), and \(\Delta t=10^{-2}\).

For the strongly convex test, we use
\begin{align*}
    f_{\mathrm{PL}}(W)
    =\frac12\left\langle W-W^\star,(W-W^\star)H\right\rangle,
\end{align*}
where \(H=Q\Lambda Q^\top\), \(Q\) is generated from a Gaussian matrix,
and the eigenvalues are logarithmically spaced between \(\mu=1\) and
\(L=20\). Using random seed \(0\), \(W^\star\) has standard Gaussian
entries and \(W_0=W^\star+3Z\), where \(Z_{ij}\sim\mathcal N(0,1)\); the
trajectory is integrated to \(T=30\). Writing
\(E_0=f_{\mathrm{PL}}(W_0)\), we use
\(K_G=L\sqrt{2E_0/\mu}\) and
\(c_\epsilon=1/\sqrt{K_G^2+\epsilon}\).

For the general convex test, we use
\begin{align*}
    f_{\mathrm{cvx}}(W)
    =\sum_{i,j}\left(\sqrt{1+(W_{ij}-W^\star_{ij})^2}-1\right).
\end{align*}
Using random seed \(10\), the initial condition is
\(W_0=W^\star+5Z\), and the trajectory is integrated to \(T=40\). For
\(E_0=f_{\mathrm{cvx}}(W_0)\), the constants are chosen as
\(R_0=\sqrt{mn}\sqrt{(E_0+1)^2-1}\), \(K_G=\sqrt{mn}\), and
\(c_\epsilon=1/\sqrt{K_G^2+\epsilon}\).

The corresponding theoretical bounds are
\begin{align}
    E(t)
    &\leq
    \mathcal B_{\mathrm{PL}}(t)
    :=
    E(0)e^{-2\mu c_\epsilon t},
    \label{eq:numerical_pl_bound}
    \\
    E(t)
    &\leq
    \mathcal B_{\mathrm{cvx}}(t)
    :=
    \frac{R_0^2E(0)}
    {R_0^2+c_\epsilon tE(0)}.
    \label{eq:numerical_general_convex_bound}
\end{align}
In each case, \(c_\epsilon\) is determined from a uniform gradient bound
on the initial sublevel set.

For each trajectory, we plot
\begin{align}
    \mathcal Q(t)=\frac{E(t)}{\mathcal B(t)}.
    \label{eq:numerical_bound_ratio}
\end{align}
The theoretical estimate is satisfied whenever
\(\mathcal Q(t)\leq1\). Ratios are shown on a logarithmic vertical scale
and truncated once the observed gap falls below \(10^{-12}E_0\), avoiding
values dominated by floating-point roundoff.

\begin{figure*}[t]
    \centering
    \includegraphics[width=0.82\textwidth]
    {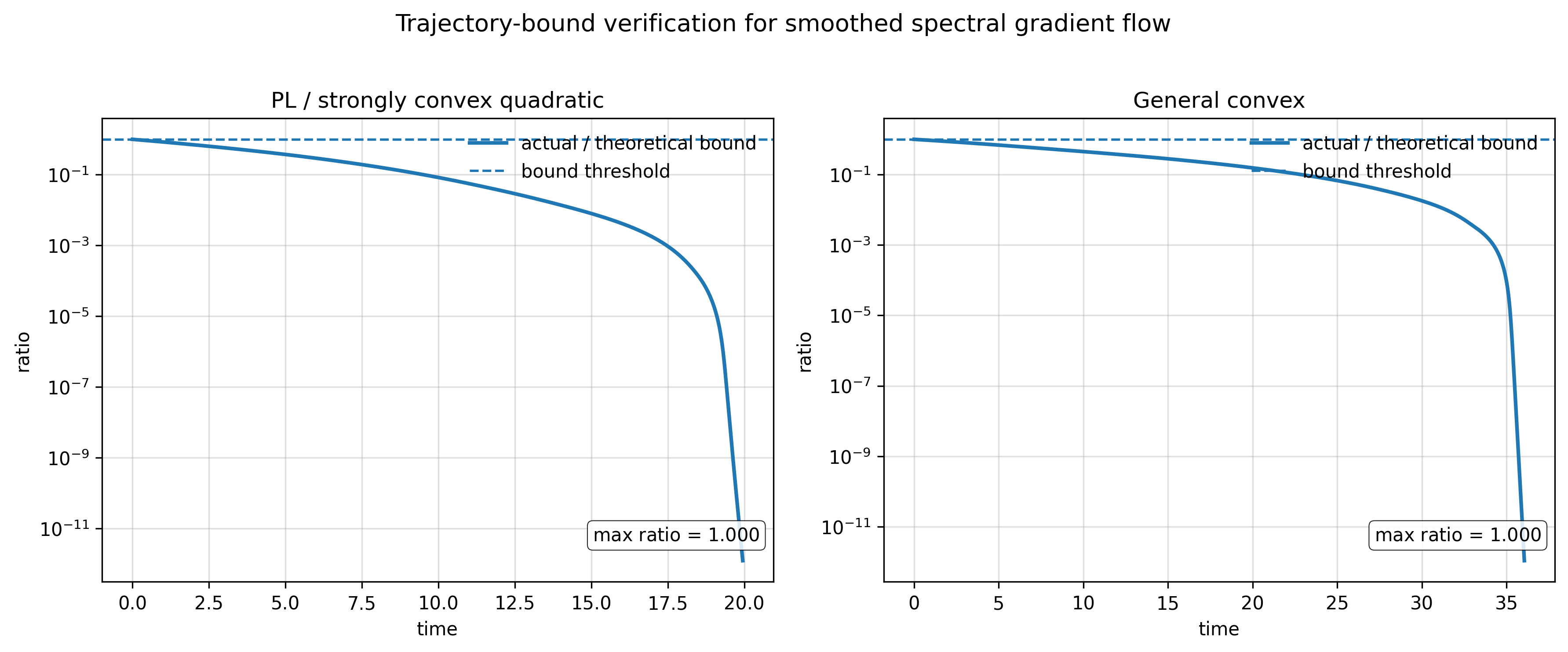}
    \caption{
    Ratio between the observed objective gap and the theoretical
    trajectory bound. Left: strongly convex quadratic loss.
    Right: general convex loss. The dashed line denotes
    \(\mathcal Q(t)=1\).
    }
    \label{fig:trajectory_bound_verification}
\end{figure*}

Figure~\ref{fig:trajectory_bound_verification} shows that the observed
ratios remain below one over the numerically resolved trajectories,
supporting both the exponential PL bound and the sublinear convex bound.
The bounds are conservative because their constants are obtained from
uniform sublevel-set estimates.

\section{Conclusion}
\label{sec:conclusion}

This paper studied a class of spectral gradient flows induced by smoothed
matrix-polar feedback. The flow uses
\[
h_\epsilon(G)
=
G(G^\top G+\epsilon I)^{-1/2}
\]
as the gradient-feedback direction. This map preserves the directional
structure of the matrix-polar factor, while the smoothing parameter
$\epsilon>0$ removes the discontinuity at singular points. We first
established the regularity, monotonicity, and boundedness properties of
this feedback map and then proved the global well-posedness of the resulting
smoothed spectral gradient flow. Using the associated energy-dissipation
structure, we further established global existence, vanishing of the
gradient, and convergence to the stationary set.

At the global level, we established integral gradient bounds,
nonconvex stationarity rates, convex objective-gap estimates, and
exponential convergence under the PL condition for the direct
smoothed spectral flow. For the momentum-augmented system, we
derived a compatible Lyapunov structure together with convergence
to the stationary set and a stationarity-rate estimate. Under a uniform
positive lower bound on the relevant
singular values, we further proved that the smoothed direct and
momentum-augmented flows approximate their ideal matrix-polar
counterparts with an $\mathcal{O}(\epsilon)$ error on every fixed
finite time interval, providing a finite-horizon connection to
the Muon matrix-polar dynamics. At the
local level, a common local quadratic model yielded the
normalized descent-rate ratio $\Gamma_\epsilon(W)$ for comparing
the smoothed spectral and Frobenius-gradient directions. The
criterion shows that any local advantage of spectral normalization
is conditional on the interaction between the singular-value
structure of the gradient and the directional curvature induced by the
local Hessian.

Several directions remain open. First, it remains to construct a
continuous-time model that remains valid across rank changes and to
quantify how such a model approximates discrete Muon-type algorithms with finite
stepsizes, stochastic gradients, and practical momentum recursions.
Second, an adaptive choice of $\epsilon$ may provide a better balance
between regularity and the strength of spectral normalization. Third,
extending the local spectral-structure analysis to deeper network
architectures may clarify the favorable regimes and failure modes of
matrix-polar-type updates in practical training.

\section*{References}
\bibliographystyle{IEEEtran}
\bibliography{reference}

\appendices
\section{Proofs of Spectral Feedback Properties}
\label{app:heps_properties}

\begin{proof}
Let
\[
A(M)=M^\top M+\epsilon I.
\]
Since $\epsilon>0$, $A(M)$ is positive definite. For any perturbation
$H\in\mathbb R^{m\times n}$, the differential formula for a spectral trace
function gives
\[
d\,\operatorname{Tr}(A^{1/2})
=
\frac12\operatorname{Tr}(A^{-1/2}dA).
\]
Consequently,
\[
\begin{aligned}
d\Phi_\epsilon(M)[H]
&=
\frac12\operatorname{Tr}
\left(
A(M)^{-1/2}(M^\top H+H^\top M)
\right)\\
&=
\left\langle
M A(M)^{-1/2},H
\right\rangle .
\end{aligned}
\]
Thus
\[
\nabla_M\Phi_\epsilon(M)
=
M(M^\top M+\epsilon I)^{-1/2}
=
h_\epsilon(M),
\]
which proves \eqref{eq:Phi_gradient}.

For monotonicity, define
\[
\phi_\epsilon(s)=\sqrt{s^2+\epsilon}-\sqrt{\epsilon}.
\]
The function $\phi_\epsilon$ is even and convex because
\[
\phi_\epsilon''(s)
=
\frac{\epsilon}{(s^2+\epsilon)^{3/2}}
\ge0.
\]
Hence
\[
\Phi_\epsilon(M)
=
\sum_{i=1}^{d}\phi_\epsilon(\sigma_i(M))
\]
is a convex unitarily invariant spectral function. Since
$h_\epsilon=\nabla\Phi_\epsilon$, convexity yields
\[
\left\langle
h_\epsilon(M)-h_\epsilon(N),M-N
\right\rangle
\ge0,
\]
which is \eqref{eq:heps_monotone}.

To establish the Lipschitz estimate, let
\[
g_\epsilon(s)=\frac{s}{\sqrt{s^2+\epsilon}}.
\]
Then
\[
|g_\epsilon'(s)|
=
\frac{\epsilon}{(s^2+\epsilon)^{3/2}}
\le
\frac1{\sqrt{\epsilon}}.
\]
Introduce the symmetric dilation
\[
\mathcal T(M)
=
\begin{pmatrix}
0&M\\
M^\top&0
\end{pmatrix}.
\]
Functional calculus gives
\[
g_\epsilon(\mathcal T(M))
=
\begin{pmatrix}
0&h_\epsilon(M)\\
h_\epsilon(M)^\top&0
\end{pmatrix}.
\]
The Frobenius-norm spectral-operator estimate therefore implies
\[
\begin{aligned}
\sqrt2\|h_\epsilon(M)-h_\epsilon(N)\|_F
&\le
\frac1{\sqrt{\epsilon}}
\|\mathcal T(M)-\mathcal T(N)\|_F\\
&=
\frac{\sqrt2}{\sqrt{\epsilon}}\|M-N\|_F.
\end{aligned}
\]
This proves \eqref{eq:heps_lipschitz}. Continuous differentiability follows
from the smoothness of the inverse square-root map on the positive-definite cone.

Finally, let $M=U\Sigma V^\top$ be a singular value decomposition. Then
\[
h_\epsilon(M)
=
U\operatorname{diag}
\left(
\frac{\sigma_i}{\sqrt{\sigma_i^2+\epsilon}}
\right)V^\top .
\]
It follows that
\[
\|h_\epsilon(M)\|_F^2
=
\sum_{i=1}^{d}
\frac{\sigma_i^2}{\sigma_i^2+\epsilon}
\le d,
\]
which proves \eqref{eq:heps_bound}. Moreover,
\[
\left\langle h_\epsilon(M),M\right\rangle
=
\sum_{i=1}^{d}
\frac{\sigma_i^2}{\sqrt{\sigma_i^2+\epsilon}}.
\]
The upper estimate in \eqref{eq:heps_dissipation_bounds} follows from
$\sqrt{\sigma_i^2+\epsilon}\ge\sqrt{\epsilon}$. If $\|M\|_F\le K$, then
$\sigma_i\le K$ and
$\sqrt{\sigma_i^2+\epsilon}\le\sqrt{K^2+\epsilon}$, which gives the lower
estimate after summation.
\end{proof}

\section{Proofs for the Direct Smoothed Spectral Gradient Flow}
\label{app:direct_flow_proofs}

\subsection{Global Well-Posedness}
\label{app:direct_wellposedness}

\begin{proof}
Define \(F(W)=-h_\epsilon(\nabla f(W))\). Under Assumption~\ref{ass:local_gradient_lipschitz},
$\nabla f$ is locally Lipschitz, while
$h_\epsilon$ is globally Lipschitz by Lemma~\ref{lem:heps_properties}.
Hence $F$ is locally Lipschitz and the Picard--Lindel\"of theorem gives a
unique maximal classical solution.

The uniform feedback bound gives
\(\|\dot W(t)\|_F=\|h_\epsilon(\nabla f(W(t)))\|_F\le\sqrt d\), hence
\(\|W(t)-W_0\|_F\le\sqrt d\,t\). Thus the state cannot escape to infinity
in finite time, and the ODE continuation theorem extends the solution to
all $t\ge0$.
\end{proof}

\subsection{Convergence to the Stationary Set}
\label{app:direct_convergence}

\begin{proof}
Let \(G(t)=\nabla f(W(t))\). Since
\(\frac{d}{dt}f(W(t))=-\langle G(t),h_\epsilon(G(t))\rangle\le0\), the
trajectory remains in the bounded sublevel set
\(\mathcal L_0=\{W:f(W)\le f(W_0)\}\). Hence
\(K_G=\sup_{W\in\mathcal L_0}\|\nabla f(W)\|_F<\infty\). With
\(c_\epsilon=1/\sqrt{K_G^2+\epsilon}\), Lemma~\ref{lem:heps_properties}
implies
\[
\left\langle
G(t),h_\epsilon(G(t))
\right\rangle
\ge
c_\epsilon\|G(t)\|_F^2.
\]
Hence
\(\frac{d}{dt}f(W(t))\le-c_\epsilon\|\nabla f(W(t))\|_F^2\). Integration
and the lower bound $f_{\inf}$ yield
\[
\int_0^\infty
\|\nabla f(W(t))\|_F^2\,dt
\le
\frac{f(W_0)-f_{\inf}}{c_\epsilon}
<\infty.
\]

On the compact set $\mathcal L_0$, the gradient has a Lipschitz constant
$L_0$. Since $\|\dot W(t)\|_F\le\sqrt d$, the map
$t\mapsto\nabla f(W(t))$ is uniformly continuous; Barbalat's lemma gives
\(\nabla f(W(t))\to0\).
If $\bar W\in\omega(W_0)$, then there exists a sequence $t_k\to\infty$ such that $W(t_k)\to\bar W$.
Continuity of $\nabla f$ gives $\nabla f(\bar W)=0$, and therefore
\(\omega(W_0)\subseteq\mathcal S\).
\end{proof}

\section{Proofs for the Momentum-Augmented System}
\label{app:momentum_flow_proofs}

\subsection{Global Convergence}
\label{app:momentum_convergence}

\begin{proof}
The vector field in \eqref{eq:momentum_augmented_flow} is locally Lipschitz
under Assumption~\ref{ass:local_gradient_lipschitz}. Hence a unique
maximal solution exists. On every finite interval, \(\|\dot W(t)\|_F\le\sqrt d\),
so $W(t)$ is bounded. Boundedness of $W$ and local Lipschitz continuity of
$\nabla f$ imply boundedness of the gradient on that interval. The linear
momentum equation and Gronwall's inequality then bound $M(t)$. Thus
finite-time escape is impossible and the solution is global.

Along the trajectory, the storage function \eqref{eq:momentum_storage}
satisfies \eqref{eq:momentum_storage_dissipation}; hence
\(V(W(t),M(t))\le V_0\). The inequality
\(a(f(W(t))-f_{\inf})\le V_0\), together with coercivity, bounds $W(t)$.
Moreover, \(\Phi_\epsilon(M(t))\le V_0\). If
$\sigma_i$ are the singular values of $M$, then
\[
\Phi_\epsilon(M)
\ge
\sum_i\sigma_i-d\sqrt{\epsilon}
=
\|M\|_*-d\sqrt{\epsilon}.
\]
Therefore $M(t)$ is bounded as well.

The zero-dissipation set is characterized by
\(\langle h_\epsilon(M),M\rangle=0\) if and only if \(M=0\). If a complete
trajectory remains in this set, then \(\dot M=0\), so the momentum equation
gives \(\nabla f(W)=0\), and also \(\dot W=-h_\epsilon(0)=0\). Thus the
largest invariant subset is \(\mathcal S\times\{0\}\), and LaSalle's
invariance principle \cite{lasalle1961stability,khalil2002nonlinear} gives
\(\omega(W_0,M_0)\subseteq\mathcal S\times\{0\}\).
\end{proof}

\subsection{Momentum-Augmented System and Velocity Bounds}
\label{app:momentum_trajectory_certificates}
\label{app:momentum_trajectory_bounds}

\begin{proof}
The boundedness established in Appendix~\ref{app:momentum_convergence}
ensures that \(K_M=\sup_{t\ge0}\|M(t)\|_F<\infty\).
Lemma~\ref{lem:heps_properties} and the storage identity give
\[
\dot V(t)
\le
-\frac{b}{\sqrt{K_M^2+\epsilon}}\|M(t)\|_F^2.
\]
Integration yields
\[
\int_0^\infty\|M(t)\|_F^2\,dt
\le
\frac{\sqrt{K_M^2+\epsilon}}{b}V_0.
\]
Furthermore,
\[
\|\dot W(t)\|_F^2
=
\|h_\epsilon(M(t))\|_F^2
\le
\frac1{\sqrt{\epsilon}}
\langle h_\epsilon(M(t)),M(t)\rangle.
\]
Using
$\dot V=-b\langle h_\epsilon(M),M\rangle$ and integrating gives
\[
\int_0^\infty\|\dot W(t)\|_F^2\,dt
\le
\frac{V_0}{b\sqrt{\epsilon}}.
\]
\end{proof}

\subsection{Nonconvex Stationarity}
\label{app:momentum_nonconvex}

\begin{proof}
Boundedness of $(W(t),M(t))$ implies that $Q(W(t),M(t))$ is bounded, so
$\Delta_Q<\infty$. Since $\nabla f$ is Lipschitz on $\mathcal K$, the map
$t\mapsto\nabla f(W(t))$ is absolutely continuous and
\(\|\frac{d}{dt}\nabla f(W(t))\|_F\le L_{\mathcal K}\|\dot W(t)\|_F\) for
a.e. $t$. Using $\dot W=-h_\epsilon(M)$ and
\(\|h_\epsilon(M)\|_F\le\|M\|_F/\sqrt{\epsilon}\) gives
\[
\left|
\left\langle
\frac{d}{dt}\nabla f(W(t)),M(t)
\right\rangle
\right|
\le
\frac{L_{\mathcal K}}{\sqrt{\epsilon}}\|M(t)\|_F^2.
\]
Moreover,
\[
\left\langle\nabla f(W),\dot M\right\rangle
=
a\|\nabla f(W)\|_F^2
-b\langle\nabla f(W),M\rangle.
\]
Young's inequality with parameter $a/2$ yields
\[
b|\langle\nabla f(W),M\rangle|
\le
\frac a2\|\nabla f(W)\|_F^2
+
\frac{b^2}{2a}\|M\|_F^2.
\]
Therefore,
\[
\dot Q(t)
\ge
\frac a2\|\nabla f(W(t))\|_F^2
-C_M\|M(t)\|_F^2.
\]
After integration,
\[
\frac a2
\int_0^T\|\nabla f(W(t))\|_F^2\,dt
\le
\Delta_Q
+
C_M\int_0^T\|M(t)\|_F^2\,dt.
\]
Applying Lemma~\ref{thm:momentum_flow_trajectory_bounds} gives
\[
\int_0^T\|\nabla f(W(t))\|_F^2\,dt
\le
\frac2a
\left[
\Delta_Q
+
C_M\frac{\sqrt{K_M^2+\epsilon}}{b}V_0
\right].
\]
The minimum is bounded by the time average, which proves the theorem.
\end{proof}

\subsection{Convex Averaged Estimate}
\label{app:momentum_convex}

\begin{proof}
Fix \(W^*\in\mathcal W^*\) and let
\(D_{\max}=\sup_{t\ge0}\|W(t)-W^*\|_F\).
By convexity and the momentum equation,
\[
\begin{aligned}
a\int_0^T(f(W(t))-f_*)\,dt
&\le
\int_0^T
\left\langle \dot M(t),W(t)-W^*\right\rangle dt\\
&\quad+
b\int_0^T
\left\langle M(t),W(t)-W^*\right\rangle dt\\
=
\left[\langle M,W-W^*\rangle\right]_0^T\\
&\quad-
\int_0^T\langle M,\dot W\rangle dt\\
&\quad+
b\int_0^T\langle M,W-W^*\rangle dt.
\end{aligned}
\]
The boundary term is at most $2D_{\max}K_M$. Since
$\dot W=-h_\epsilon(M)$,
\(-\int_0^T\langle M,\dot W\rangle dt=(V_0-V(T))/b\le V_0/b\). Let
\(B_M=\sqrt{K_M^2+\epsilon}V_0/b\).
By Cauchy--Schwarz and Lemma~\ref{thm:momentum_flow_trajectory_bounds},
\[
b\int_0^T\langle M,W-W^*\rangle dt
\le
bD_{\max}\sqrt{B_M T}.
\]
Hence
\(\int_0^T(f(W(t))-f_*)\,dt\le C_0+C_1\sqrt T\) for constants
$C_0,C_1$ independent of $T$. Jensen's inequality gives
\(f(\bar W_T)-f_*\le C_0/T+C_1/\sqrt T\), which is bounded by
\((C_0+C_1)/\sqrt T\) for $T\ge1$.
Taking $C_{\rm mom}=C_0+C_1$ proves the theorem.
\end{proof}

\section{Proof of Proposition~\ref{prop:finite-horizon-polar-approximation}}
\label{app:finite-horizon-polar-approximation}

\subsection{Uniform Approximation of the Feedback Law}

Let
\(Z=U\operatorname{diag}(\sigma_1,\ldots,\sigma_d)V^\top\) be a compact
singular value decomposition of a full-rank matrix $Z$. Then
\(\operatorname{Polar}(Z)=UV^\top\) and
\(h_\epsilon(Z)=U\operatorname{diag}(\sigma_i/\sqrt{\sigma_i^2+\epsilon})V^\top\).
By the unitary invariance of the Frobenius norm,
\begin{align}
    \left\|
        h_\epsilon(Z)-\operatorname{Polar}(Z)
    \right\|_F^2
    &=
    \sum_{i=1}^d
    \left(
        1-\frac{\sigma_i}{\sqrt{\sigma_i^2+\epsilon}}
    \right)^2.
    \label{eq:feedback-error-svd}
\end{align}
The scalar function
\(s\mapsto1-s/\sqrt{s^2+\epsilon}\) is nonincreasing on $(0,\infty)$.
Therefore, if
$\sigma_i\geq\underline{\sigma}$ for every $i$, then
\begin{align}
    \left\|
        h_\epsilon(Z)-\operatorname{Polar}(Z)
    \right\|_F
    &\leq
    \sqrt{d}
    \left(
        1-
        \frac{\underline{\sigma}}
        {\sqrt{\underline{\sigma}^2+\epsilon}}
    \right).
    \label{eq:feedback-error-first-bound}
\end{align}
Moreover, for every $s>0$,
\begin{align}
    1-\frac{s}{\sqrt{s^2+\epsilon}}
    &=
    \frac{
        \sqrt{s^2+\epsilon}-s
    }{
        \sqrt{s^2+\epsilon}
    }
    \notag\\
    &=
    \frac{\epsilon}{
        \sqrt{s^2+\epsilon}
        \bigl(\sqrt{s^2+\epsilon}+s\bigr)
    }
    \leq
    \frac{\epsilon}{2s^2}.
    \label{eq:scalar-polar-error}
\end{align}
Taking $s\geq\underline{\sigma}$ in
\eqref{eq:scalar-polar-error} and substituting into
\eqref{eq:feedback-error-svd} gives
\[
    \left\|
        h_\epsilon(Z)-\operatorname{Polar}(Z)
    \right\|_F
    \leq
    \frac{\sqrt{d}}{2\underline{\sigma}^2}\epsilon.
\]
This proves \eqref{eq:uniform-feedback-approximation}.

For later use, define
\begin{align}
    c_{\underline{\sigma}}
    :=
    \frac{\sqrt{d}}{2\underline{\sigma}^2}.
    \label{eq:rank-safe-feedback-constant}
\end{align}

\subsection{Approximation of the Direct Flow}

Let \(\mathcal{K}_{\mathrm{dir}}=\{W_{\mathrm P}(t):0\leq t\leq T\}\),
which is compact. The rank-safe condition
\eqref{eq:direct-rank-safe-condition}, together with continuity of
the singular values and of $\nabla f$, gives a compact tubular
neighborhood $\mathcal{U}_{\mathrm{dir}}$ of
$\mathcal{K}_{\mathrm{dir}}$ on which
\(\sigma_{\min}(\nabla f(W))\geq\underline{\sigma}\).
On this neighborhood, the ideal vector field
\(F_{\mathrm P}^{\mathrm{dir}}(W)=-\operatorname{Polar}(\nabla f(W))\)
is Lipschitz, say with constant $L_{\mathrm{dir}}$, because
$\operatorname{Polar}$ is smooth on the full-rank set and $\nabla f$ is
locally Lipschitz. The smoothed vector field
\(F_\epsilon^{\mathrm{dir}}(W)=-h_\epsilon(\nabla f(W))\)
satisfies, by \eqref{eq:uniform-feedback-approximation},
\[
    \|F_\epsilon^{\mathrm{dir}}(W)
    -F_{\mathrm P}^{\mathrm{dir}}(W)\|_F
    \leq
    c_{\underline{\sigma}}\epsilon,
    \qquad
    W\in\mathcal{U}_{\mathrm{dir}}.
\]

Let $\tau_\epsilon$ be the first exit time of $W_\epsilon(t)$
from $\mathcal{U}_{\mathrm{dir}}$, truncated at $T$, and set
\(e_{\mathrm{dir}}(t)=\|W_\epsilon(t)-W_{\mathrm P}(t)\|_F\).
For $0\leq t\leq\tau_\epsilon$, the common initial condition gives
\[
    e_{\mathrm{dir}}(t)
    \leq
    c_{\underline{\sigma}}\epsilon t
    +
    L_{\mathrm{dir}}
    \int_0^t e_{\mathrm{dir}}(s)\,ds.
\]
Thus
\[
    e_{\mathrm{dir}}(t)
    \leq
    c_{\underline{\sigma}}\epsilon
    \psi_{L_{\mathrm{dir}}}(t),
    \qquad
    \psi_L(t)=
    \begin{cases}
        (e^{Lt}-1)/L, & L>0,\\
        t, & L=0.
    \end{cases}
\]
Set \(C_T^{\mathrm{dir}}=c_{\underline{\sigma}}\psi_{L_{\mathrm{dir}}}(T)\).
Choosing $\epsilon_T^{\mathrm{dir}}$ so that
$C_T^{\mathrm{dir}}\epsilon_T^{\mathrm{dir}}$ is smaller than the
tube radius, the trajectory cannot exit the tube before time $T$.
Therefore, for $0<\epsilon\leq\epsilon_T^{\mathrm{dir}}$,
\[
    \sup_{0\leq t\leq T}
    \|W_\epsilon(t)-W_{\mathrm P}(t)\|_F
    \leq
    C_T^{\mathrm{dir}}\epsilon,
\]
which proves \eqref{eq:direct-trajectory-approximation}.

\subsection{Approximation of the Momentum-Augmented Flow}

The momentum argument proceeds analogously under the product norm
\(\|(W,M)\|_\times=\|W\|_F+\|M\|_F\). Let
\(X_{\mathrm P}(t)=(W_{\mathrm P}(t),M_{\mathrm P}(t))\). By
\eqref{eq:momentum-rank-safe-condition}, there is a compact
tubular neighborhood $\mathcal{U}_{\mathrm{mom}}$ of
\(\{X_{\mathrm P}(t):0\leq t\leq T\}\) on which
\(\sigma_{\min}(M)\geq\underline{\sigma}\).
Define
\begin{align*}
    F_{\mathrm P}^{\mathrm{mom}}(W,M)
    &=
    \begin{pmatrix}
        -\operatorname{Polar}(M)\\
        a\nabla f(W)-bM
    \end{pmatrix},
    \\
    F_\epsilon^{\mathrm{mom}}(W,M)
    &=
    \begin{pmatrix}
        -h_\epsilon(M)\\
        a\nabla f(W)-bM
    \end{pmatrix}.
\end{align*}
On $\mathcal{U}_{\mathrm{mom}}$, the ideal vector field is
Lipschitz with some constant $L_{\mathrm{mom}}$, and the two vector
fields satisfy
\[
    \|F_\epsilon^{\mathrm{mom}}(X)
    -F_{\mathrm P}^{\mathrm{mom}}(X)\|_\times
    =
    \|h_\epsilon(M)-\operatorname{Polar}(M)\|_F
    \leq
    c_{\underline{\sigma}}\epsilon.
\]
Applying the same Gronwall and no-exit argument to
\(e_{\mathrm{mom}}(t)=\|(W_\epsilon(t),M_\epsilon(t))
-(W_{\mathrm P}(t),M_{\mathrm P}(t))\|_\times\) gives, with
\(C_T^{\mathrm{mom}}=c_{\underline{\sigma}}\psi_{L_{\mathrm{mom}}}(T)\),
a constant \(\epsilon_T^{\mathrm{mom}}>0\) such that
\[
    \sup_{0\leq t\leq T}
    \left(
        \|W_\epsilon(t)-W_{\mathrm P}(t)\|_F
        +
        \|M_\epsilon(t)-M_{\mathrm P}(t)\|_F
    \right)
    \leq
    C_T^{\mathrm{mom}}\epsilon
\]
whenever \(0<\epsilon\leq\epsilon_T^{\mathrm{mom}}\). This proves
\eqref{eq:momentum-trajectory-approximation} and completes the
proof.

\section{Proof of the Local Descent-Rate Comparison}
\label{app:local_advantage}

\begin{proof}
For a direction $D$, the quadratic model along $-\alpha D$ is
\[
q_D(\alpha)
=
f(W)-\alpha\langle G,D\rangle
+
\frac{\alpha^2}{2}
\langle D,\mathcal H_W[D]\rangle.
\]
If $\langle G,D\rangle>0$ and
$\langle D,\mathcal H_W[D]\rangle>0$, differentiation gives
\[
\alpha_D^*
=
\frac{\langle G,D\rangle}
{\langle D,\mathcal H_W[D]\rangle}.
\]
The corresponding first-order descent rate is
\[
R_W(D)
=
\alpha_D^*\langle G,D\rangle
=
\frac{\langle G,D\rangle^2}
{\langle D,\mathcal H_W[D]\rangle}.
\]
Taking $D=G$ and $D=h_\epsilon(G)$ gives
\[
R_F(W)
=
\frac{\|G\|_F^4}
{\langle G,\mathcal H_W[G]\rangle}
\]
and
\[
R_\epsilon(W)
=
\frac{\langle G,h_\epsilon(G)\rangle^2}
{\left\langle
h_\epsilon(G),\mathcal H_W[h_\epsilon(G)]
\right\rangle}.
\]
Their ratio is exactly $\Gamma_\epsilon(W)$, so
$R_\epsilon(W)>R_F(W)$ if and only if $\Gamma_\epsilon(W)>1$.

For completeness, let $G=U\Sigma V^\top$. Then
\[
h_\epsilon(G)
=
U\operatorname{diag}
\left(
\frac{\sigma_i}{\sqrt{\sigma_i^2+\epsilon}}
\right)V^\top
\to
\operatorname{Polar}(G)
\]
as $\epsilon\downarrow0$, where $\operatorname{Polar}(G)$ is the
canonical polar factor on
the nonzero singular subspace. Thus
\[
\langle G,h_\epsilon(G)\rangle\to\|G\|_*
\]
and, whenever the limiting curvature is positive,
\[
R_\epsilon(W)
\to
\frac{\|G\|_*^2}
{\langle \operatorname{Polar}(G),
\mathcal H_W[\operatorname{Polar}(G)]\rangle}.
\]
\end{proof}

\subsection{Simplified Squared-Loss Model Specialization}

For
\[
f(W)=\frac12\|WA-Y\|_F^2,
\]
the gradient and Hessian action are
\[
\nabla f(W)=(WA-Y)A^\top,
\qquad
\mathcal H_W[D]=DAA^\top.
\]
Consequently,
\[
\langle D,\mathcal H_W[D]\rangle=\|DA\|_F^2,
\]
and the general local rate becomes
\[
R_W(D)
=
\frac{\langle G,D\rangle^2}{\|DA\|_F^2}.
\]
Taking $D=G$ and $D=h_\epsilon(G)$ gives
\[
R_F(W)=\frac{\|G\|_F^4}{\|GA\|_F^2},
\qquad
R_\epsilon(W)
=
\frac{\langle G,h_\epsilon(G)\rangle^2}
{\|h_\epsilon(G)A\|_F^2},
\]
and hence
\[
\Gamma_\epsilon(W)
=
\frac{
\langle G,h_\epsilon(G)\rangle^2\|GA\|_F^2
}{
\|G\|_F^4\|h_\epsilon(G)A\|_F^2
}.
\]
This expression isolates how the feature-induced Hessian action modifies the curvature
cost of the two directions.

\section{Proofs of Convergence Rate Results}
\label{app:direct_trajectory_certificates}

\subsection{Integral Trajectory Bounds}
\label{app:direct_trajectory_bounds}

\begin{proof}
The dissipation estimate established in
Appendix~\ref{app:direct_convergence} gives
\[
\frac{d}{dt}f(W(t))
\le
-c_\epsilon\|\nabla f(W(t))\|_F^2.
\]
Integration over $[0,T]$ yields
\[
\int_0^T\|\nabla f(W(t))\|_F^2\,dt
\le
\frac{f(W_0)-f_{\inf}}{c_\epsilon}.
\]

For any matrix $G$ with singular values $\sigma_i$,
\[
\begin{aligned}
\|h_\epsilon(G)\|_F^2
&=
\sum_i\frac{\sigma_i^2}{\sigma_i^2+\epsilon}\\
&\le
\frac1{\sqrt{\epsilon}}
\sum_i\frac{\sigma_i^2}{\sqrt{\sigma_i^2+\epsilon}}\\
&=
\frac1{\sqrt{\epsilon}}
\langle G,h_\epsilon(G)\rangle.
\end{aligned}
\]
Since $\dot W=-h_\epsilon(\nabla f(W))$, integration of the exact energy
identity gives
\[
\int_0^T\|\dot W(t)\|_F^2\,dt
\le
\frac{f(W_0)-f(W(T))}{\sqrt{\epsilon}}
\le
\frac{f(W_0)-f_{\inf}}{\sqrt{\epsilon}}.
\]
Letting $T\to\infty$ proves the two infinite-horizon statements.
\end{proof}

\subsection{Nonconvex Stationarity}
\label{app:direct_nonconvex}

\begin{proof}
By Lemma~\ref{thm:direct_flow_trajectory_bounds},
\[
\int_0^T\|\nabla f(W(t))\|_F^2\,dt
\le
\frac{f(W_0)-f_{\inf}}{c_\epsilon}.
\]
The minimum is bounded by the time average, so
\[
\min_{0\le t\le T}\|\nabla f(W(t))\|_F^2
\le
\frac{f(W_0)-f_{\inf}}{c_\epsilon T}.
\]
The convergence $\nabla f(W(t))\to0$ follows from
Theorem~\ref{thm:direct_flow_convergence}.
\end{proof}

\subsection{Convex Objective Gap}
\label{app:direct_convex}

\begin{proof}
Let
\[
E(t)=f(W(t))-f_*.
\]
Convexity gives, for any $W^*\in\mathcal W^*$,
\[
E(t)
\le
\left\langle
\nabla f(W(t)),W(t)-W^*
\right\rangle
\le
R_0\|\nabla f(W(t))\|_F.
\]
Together with the dissipation estimate,
\[
\dot E(t)
\le
-c_\epsilon\|\nabla f(W(t))\|_F^2
\le
-\frac{c_\epsilon}{R_0^2}E(t)^2.
\]
If $E(0)=0$, the result is immediate. Otherwise,
\[
\frac{d}{dt}\frac1{E(t)}
\ge
\frac{c_\epsilon}{R_0^2}.
\]
Integration gives
\[
E(t)
\le
\frac{R_0^2E(0)}
{R_0^2+c_\epsilon tE(0)},
\]
which is the stated estimate.
\end{proof}

\subsection{PL Exponential Estimate}
\label{app:direct_pl}

\begin{proof}
With $E(t)=f(W(t))-f_*$, the PL inequality and the dissipation estimate give
\[
\dot E(t)
\le
-c_\epsilon\|\nabla f(W(t))\|_F^2
\le
-2\mu c_\epsilon E(t).
\]
Gronwall's inequality yields
\[
E(t)
\le
e^{-2\mu c_\epsilon t}E(0).
\]
\end{proof}

\begin{IEEEbiography}[{\includegraphics[width=1in,height=1.25in,clip,keepaspectratio]{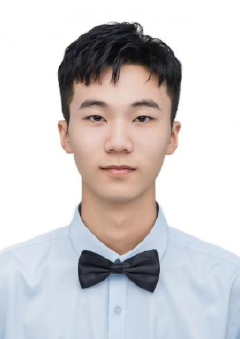}}]{Jinlin Liu}
is currently pursuing a bachelor's degree in mathematics at Zhejiang
University, Hangzhou, China. His research interests include optimization,
dynamical systems, matrix optimization, and deep learning theory, with a
particular focus on continuous-time analysis of optimization
methods.
\end{IEEEbiography}

\begin{IEEEbiography}[{\includegraphics[width=1in,height=1.25in,clip,keepaspectratio]{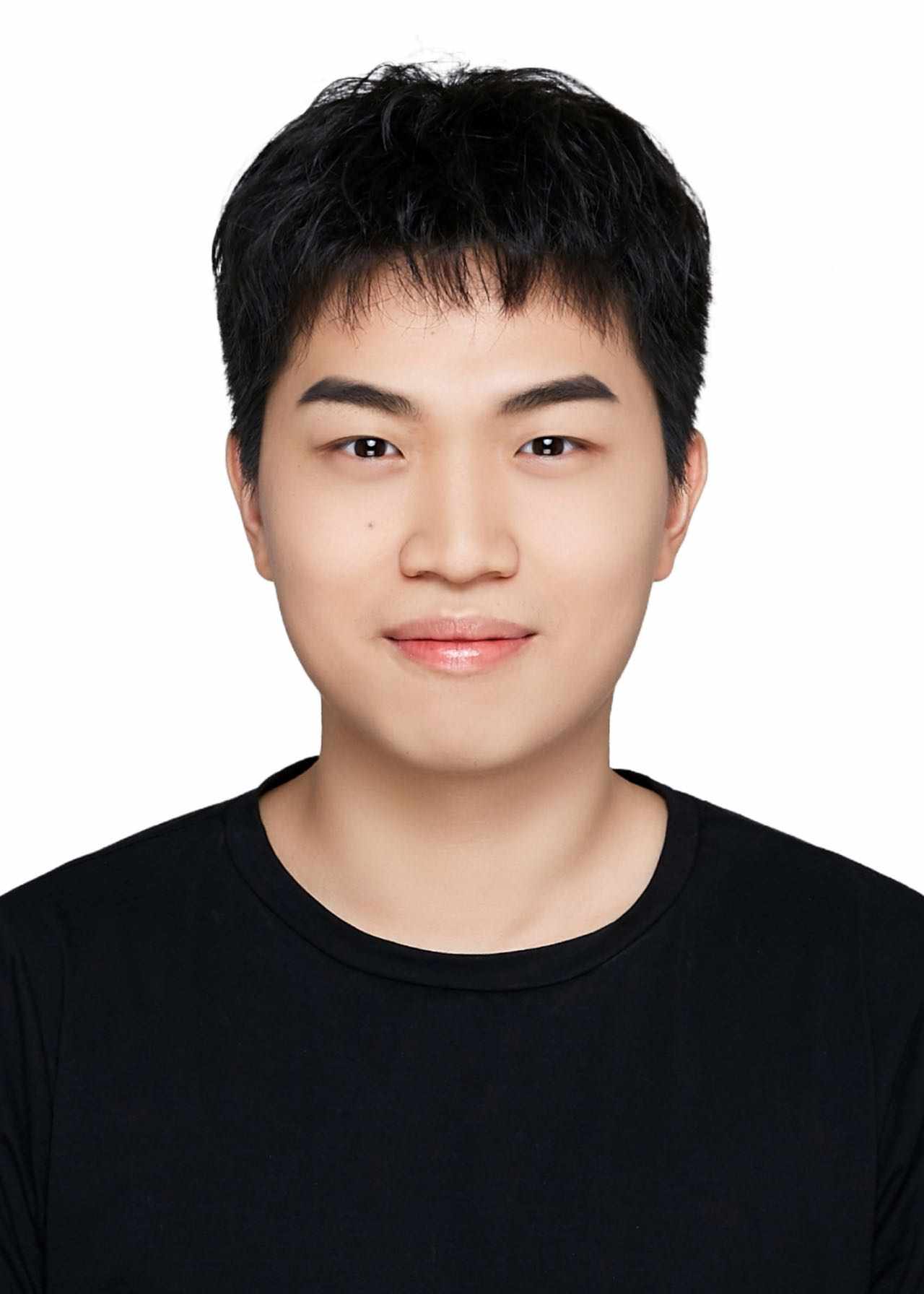}}]{Song Chen}
received the Ph.D. degree in operational research and cybernetics from Zhejiang University, Hangzhou, China, in 2025.

He is currently a Research Fellow with the Department of Mathematics, National University of Singapore (NUS), Singapore. His research interests lie at the intersection of control theory and artificial intelligence, with a particular focus on control-oriented learning methods and embodied AI. His broader expertise includes convex optimization, nonlinear control, and machine learning theory with applications in robotics.
\end{IEEEbiography}

\begin{IEEEbiography}[{\includegraphics[width=1in,height=1.25in,clip,keepaspectratio]{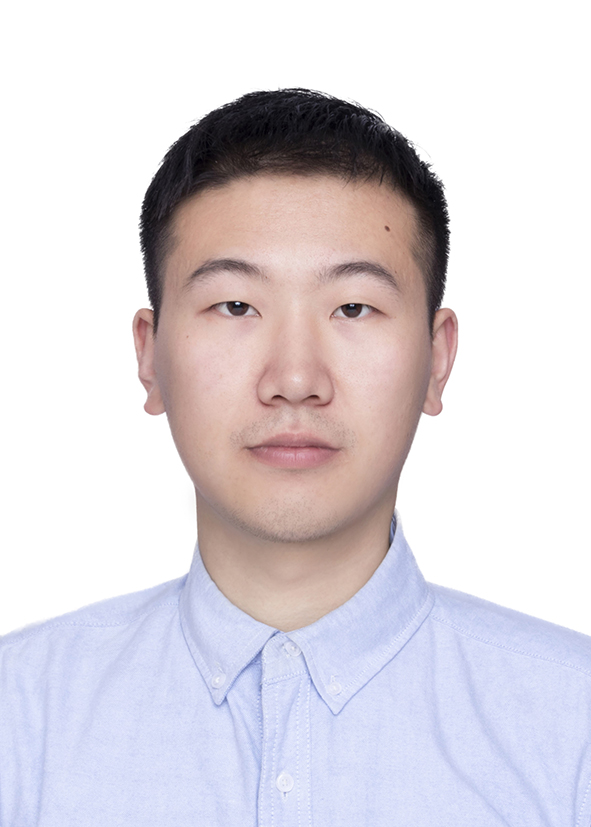}}]{Jiaxu Liu}
received a bachelor's degree in mathematics from Renmin University of China, Beijing, China, in 2021. He received the Ph.D. degree in operational research and cybernetics from Zhejiang University, Hangzhou, China, in 2026. He is currently an associate professor in the School of Science at Huzhou Normal University.

His research interests cover distributed optimization, convex optimization, robust control, federated learning theory, and their practical applications in robotic systems.
\end{IEEEbiography}

\begin{IEEEbiography}[{\includegraphics[width=1in,height=1.25in,clip,keepaspectratio]{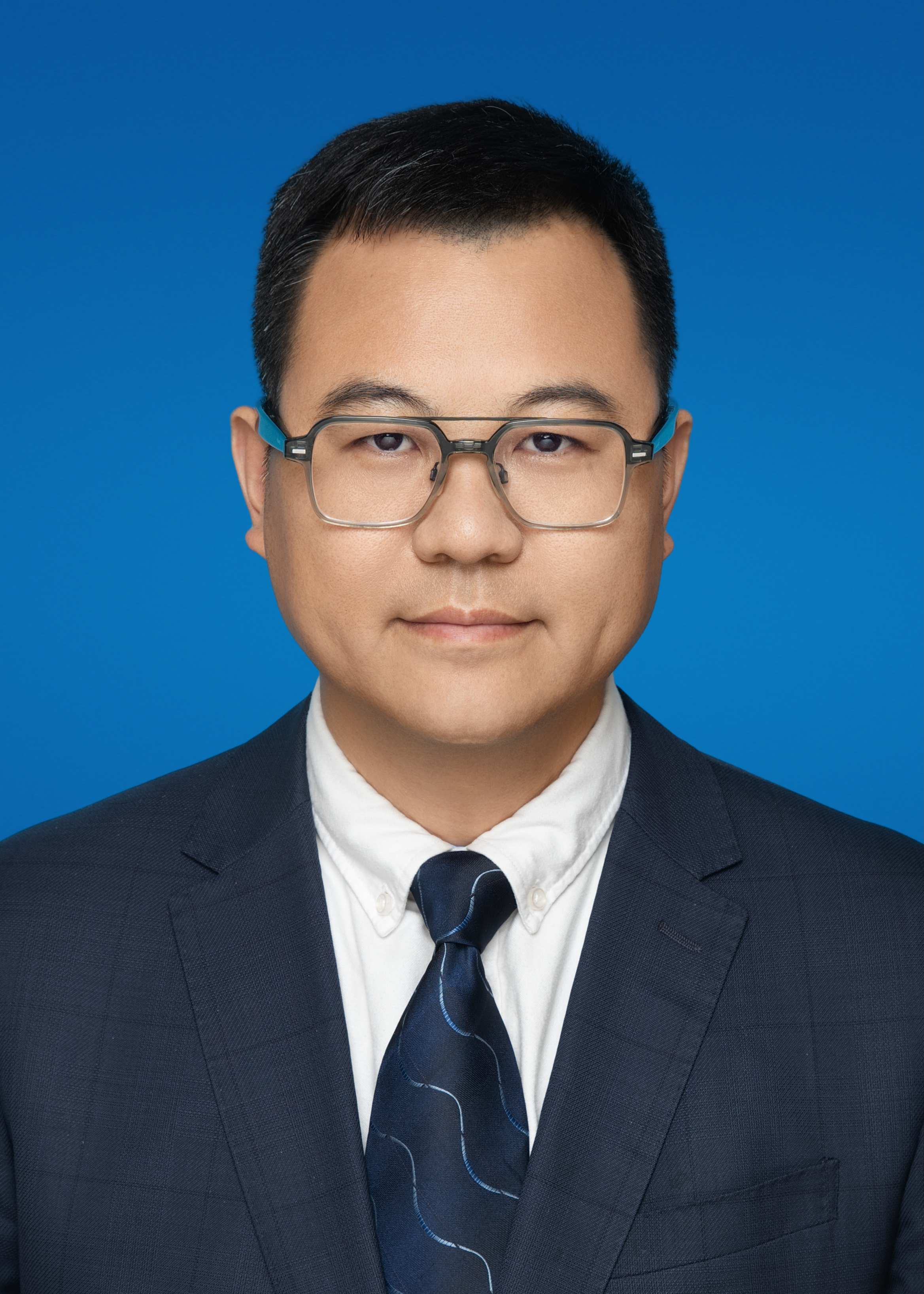}}]{Chao Xu}
(Senior Member, IEEE) received the Ph.D. degree in mechanical engineering from Lehigh University, Bethlehem, PA, USA, in 2010.

He is currently a Professor of Controls and Autonomous Systems with the College of Control Science and Engineering, Zhejiang University (ZJU). He serves as the inaugural Dean of the ZJU Huzhou Institute and as Managing Editor of two international journals: \textit{IET Cyber-Systems and Robotics} (IET-CSR) and the \textit{Journal of Industrial and Management Optimization} (JIMO). His research interests broadly encompass cybernetic physics and autonomous mobility, with a focus on modeling and control of aerial robots, machine learning for dynamical systems and control, and visual sensing and machine learning for complex fluids.
\end{IEEEbiography}

\end{document}